\pdfoutput=1
\newif\ifpersonal
\documentclass[12pt,a4paper,reqno]{amsart} %reqno places equation numbers on the right
\usepackage{amsmath,amsthm,amssymb,mathrsfs,mathtools,bm,eucal,tensor} % math related
\usepackage{microtype,lmodern} % latex technical issues
\usepackage[utf8]{inputenc} % input encoding
\usepackage[T1]{fontenc} % font encoding
\usepackage{enumerate,comment,braket,xspace,tikz-cd,csquotes} % utilities
\usepackage[centering,vscale=0.7,hscale=0.7]{geometry}
\usepackage[hidelinks]{hyperref}
\usepackage[capitalize]{cleveref}

\numberwithin{equation}{section}
\theoremstyle{plain}
\newtheorem{thm}[equation]{Theorem}
\newtheorem{lem}[equation]{Lemma}
\newtheorem{prop}[equation]{Proposition}

\newtheorem{cor}[equation]{Corollary}

\theoremstyle{definition}
\newtheorem{defin}[equation]{Definition}

\ifpersonal
\newcommand{\personal}[1]{\textcolor[rgb]{0,0,1}{(Personal: #1)}}
\newcommand{\todo}[1]{\textcolor{red}{(Todo: #1)}}
\else
\newcommand*{\personal}[1]{\ignorespaces}
\newcommand*{\todo}[1]{\ignorespaces}
\fi

\newcommand{\Z}{\mathbb Z}

\newcommand{\rH}{\mathrm H}

\newcommand{\rR}{\mathrm R}

\newcommand{\cC}{\mathcal C}
\newcommand{\cD}{\mathcal D}
\newcommand{\cE}{\mathcal E}
\newcommand{\cF}{\mathcal F}
\newcommand{\cH}{\mathcal H}
\newcommand{\cG}{\mathcal G}

\newcommand{\cO}{\mathcal O}
\newcommand{\cP}{\mathcal P}

\newcommand{\cS}{\mathcal S}
\newcommand{\cT}{\mathcal T}

\newcommand{\cX}{\mathcal X}
\newcommand{\cY}{\mathcal Y}

\DeclareFontFamily{U}{BOONDOX-calo}{\skewchar\font=45 }
\DeclareFontShape{U}{BOONDOX-calo}{m}{n}{<-> s*[1.05] BOONDOX-r-calo}{}
\DeclareFontShape{U}{BOONDOX-calo}{b}{n}{<-> s*[1.05] BOONDOX-b-calo}{}
\DeclareMathAlphabet{\mathcalboondox}{U}{BOONDOX-calo}{m}{n}
\newcommand{\bbA}{\mathbb A}

\newcommand{\bA}{\mathbf A}
\newcommand{\bD}{\mathbf D}
\newcommand{\bP}{\mathbf P}

\makeatletter
\let\save@mathaccent\mathaccent
\newcommand*\if@single[3]{%
	\setbox0\hbox{${\mathaccent"0362{#1}}^H$}%
	\setbox2\hbox{${\mathaccent"0362{\kern0pt#1}}^H$}%
	\ifdim\ht0=\ht2 #3\else #2\fi
}
\newcommand*\rel@kern[1]{\kern#1\dimexpr\macc@kerna}
\newcommand*\widebar[1]{\@ifnextchar^{{\wide@bar{#1}{0}}}{\wide@bar{#1}{1}}}
\newcommand*\wide@bar[2]{\if@single{#1}{\wide@bar@{#1}{#2}{1}}{\wide@bar@{#1}{#2}{2}}}
\newcommand*\wide@bar@[3]{%
	\begingroup
	\def\mathaccent##1##2{%
		\let\mathaccent\save@mathaccent
		\if#32 \let\macc@nucleus\first@char \fi
		\setbox\z@\hbox{$\macc@style{\macc@nucleus}_{}$}%
		\setbox\tw@\hbox{$\macc@style{\macc@nucleus}{}_{}$}%
		\dimen@\wd\tw@
		\advance\dimen@-\wd\z@
		\divide\dimen@ 3
		\@tempdima\wd\tw@
		\advance\@tempdima-\scriptspace
		\divide\@tempdima 10
		\advance\dimen@-\@tempdima
		\ifdim\dimen@>\z@ \dimen@0pt\fi
		\rel@kern{0.6}\kern-\dimen@
		\if#31
		\overline{\rel@kern{-0.6}\kern\dimen@\macc@nucleus\rel@kern{0.4}\kern\dimen@}%
		\advance\dimen@0.4\dimexpr\macc@kerna
		\let\final@kern#2%
		\ifdim\dimen@<\z@ \let\final@kern1\fi
		\if\final@kern1 \kern-\dimen@\fi
		\else
		\overline{\rel@kern{-0.6}\kern\dimen@#1}%
		\fi
	}%
	\macc@depth\@ne
	\let\math@bgroup\@empty \let\math@egroup\macc@set@skewchar
	\mathsurround\z@ \frozen@everymath{\mathgroup\macc@group\relax}%
	\macc@set@skewchar\relax
	\let\mathaccentV\macc@nested@a
	\if#31
	\macc@nested@a\relax111{#1}%
	\else
	\def\gobble@till@marker##1\endmarker{}%
	\futurelet\first@char\gobble@till@marker#1\endmarker
	\ifcat\noexpand\first@char A\else
	\def\first@char{}%
	\fi
	\macc@nested@a\relax111{\first@char}%
	\fi
	\endgroup
}
\makeatother

\newcommand{\PSh}{\mathrm{PSh}}
\newcommand{\Sh}{\mathrm{Sh}}

\newcommand{\Geom}{\mathrm{Geom}}

\newcommand{\infcat}{$\infty$-category\xspace}
\newcommand{\infcats}{$\infty$-categories\xspace}

\newcommand{\infsite}{$\infty$-site\xspace}

\newcommand{\inftopos}{$\infty$-topos\xspace}
\newcommand{\inftopoi}{$\infty$-topoi\xspace}

\newcommand{\rSet}{\mathrm{Set}}
\newcommand{\Ab}{\mathrm{Ab}}
\newcommand{\DAb}{\cD(\Ab)}

\newcommand{\tauet}{\tau_\mathrm{\acute{e}t}}

\newcommand{\bPsm}{\bP_\mathrm{sm}}

\newcommand{\Mod}{\textrm{-}\mathrm{Mod}}

\newcommand{\Coh}{\mathrm{Coh}}
\newcommand{\Cohb}{\mathrm{Coh}^\mathrm{b}}
\newcommand{\Cohh}{\mathrm{Coh}^\heartsuit}

\newcommand{\St}{\mathrm{St}}

\newcommand{\An}{\mathrm{An}}
\newcommand{\Afd}{\mathrm{Afd}}
\newcommand{\Top}{\mathcal T\mathrm{op}}

\newcommand{\dAn}{\mathrm{dAn}}

\newcommand{\dAnk}{\mathrm{dAn}_k}
\newcommand{\Ank}{\mathrm{An}_k}

\newcommand{\cTank}{\cT_{\mathrm{an}}(k)}

\newcommand{\cTdisck}{\cT_{\mathrm{disc}}(k)}

\newcommand{\Strloc}{\mathrm{Str}^\mathrm{loc}}
\newcommand{\RTop}{\tensor*[^\rR]{\Top}{}}

\newcommand{\Tor}{\mathrm{Tor}}
\newcommand{\dAfd}{\mathrm{dAfd}}
\newcommand{\dAfdk}{\mathrm{dAfd}_k}

\newcommand{\trunc}{\mathrm{t}_0}

\newcommand{\CRing}{\mathrm{CRing}}

\newcommand{\Cat}{\mathrm{Cat}}

\newcommand{\fib}{\mathrm{fib}}
\newcommand{\DerAn}{\mathrm{Der}\an}
\newcommand{\anL}{\mathbb L\an}

\newcommand{\bfMap}{\mathbf{Map}}
\newcommand{\cHom}{\cH \mathrm{om}}
\newcommand{\bfHom}{\mathbf{Hom}}

\newcommand{\Perf}{\mathrm{Perf}}

\newcommand{\an}{^\mathrm{an}}
\newcommand{\alg}{^\mathrm{alg}}

\newcommand{\inv}{^{-1}}

\newcommand{\kanal}{$k$-analytic\xspace}

\newcommand{\op}{^\mathrm{op}}

\usetikzlibrary{decorations.markings} %arrows for open immersions and closed immersions
\tikzset{
  closed/.style = {decoration = {markings, mark = at position 0.5 with { \node[transform shape, xscale = .8, yscale=.4] {/}; } }, postaction = {decorate} },
  open/.style = {decoration = {markings, mark = at position 0.5 with { \node[transform shape, scale = .7] {$\circ$}; } }, postaction = {decorate} }
}

\DeclareMathOperator{\Fun}{Fun}

\DeclareMathOperator{\Hom}{Hom}

\DeclareMathOperator{\Map}{Map}

\DeclareMathOperator{\Sp}{Sp}

\DeclareMathOperator*{\colim}{colim}

\DeclareMathOperator*{\cotimes}{\widehat{\otimes}}

\begin{document}
\title{Derived Hom spaces in rigid analytic geometry}

\author{Mauro PORTA}
\address{Mauro PORTA, Institut de Recherche Mathématique Avancée, 7 Rue René Descartes, 67000 Strasbourg, France}
\email{porta@math.unistra.fr}

\author{Tony Yue YU}
\address{Tony Yue YU, Laboratoire de Mathématiques d'Orsay, Université Paris-Sud, 91405 Orsay, France}
\email{yuyuetony@gmail.com}
\date{January 23, 2018}


\begin{abstract}
	We construct a derived enhancement of Hom spaces between rigid analytic spaces.
	It encodes the hidden deformation-theoretic informations of the underlying classical moduli space.
	The main tool in our construction is the representability theorem in derived analytic geometry, which has been established in our previous work.
	The representability theorem provides us sufficient and necessary conditions for an analytic moduli functor to possess the structure of a derived analytic stack.
	In order to verify the conditions of the representability theorem, we prove several general results in the context of derived non-archimedean analytic geometry: derived Tate acyclicity, projection formula, and proper base change.
	These results also deserve independent interest themselves.
	Our main motivation comes from non-archimedean enumerative geometry.
	In our subsequent works, we will apply the derived mapping stacks to obtain non-archimedean analytic Gromov-Witten invariants.
\end{abstract}

\maketitle

\tableofcontents

\section{Introduction}

Fix a complete non-archimedean field $k$ with nontrivial valuation.
Let $S$ be a rigid \kanal space, and $X, Y$ rigid \kanal spaces over $S$.
Assume that $X$ is proper and flat over $S$, and that $Y$ is separated over $S$.
Let $\An_S$ denote the category of rigid \kanal spaces over $S$.
Let $\rSet$ denote the category of sets.
Consider the functor
\begin{align*}
	\bfHom_S(X,Y)\colon\An_S &\longrightarrow\rSet\\
	T &\longmapsto\Hom_T(X_T,Y_T),
\end{align*}
where $\Hom_T(X_T,Y_T)$ denotes the set of $T$-morphisms from $X_T\coloneqq X\times_S T$ to $Y_T\coloneqq Y\times_S T$.

Conrad and Gabber proved that the functor $\bfHom_S(X,Y)$ is representable by a rigid \kanal space separated over $S$ (cf.\ \cite[Proposition 5.3.3]{Conrad_Spreading-out}).
In this paper, we construct a derived enhancement of the space $\bfHom_S(X,Y)$.
It encodes the hidden deformation-theoretic informations of the underlying classical moduli space.
Here is the precise statement:

Let $\dAn_S$ denote the \infcat of derived (rigid\footnote{When there is no confusion, we will omit the word ``rigid'' in this paper.}) \kanal spaces over $S$.
Let $\cS$ denote the \infcat of topological spaces.
Consider the $\infty$-functor
\begin{align*}
	\bfMap_S(X,Y)\colon\dAn_S &\longrightarrow\cS\\
	T &\longmapsto\Map_T(X_T,Y_T),
\end{align*}
where $\Map_T(X_T,Y_T)$ denotes the mapping space between $X_T$ and $Y_T$ in the \infcat of derived \kanal spaces over $T$.

\begin{thm}[cf.\ \cref{thm:Map}] \label{thm:Map_intro}
	The $\infty$-functor $\bfMap_S(X,Y)$ is representable by a derived \kanal space.
\end{thm}

The main tool in the proof of \cref{thm:Map_intro} is the representability theorem in derived analytic geometry, which has been established in our previous work \cite{Porta_Yu_Representability}.
Let us recall the statement:

\begin{thm}[Representability, cf.\ {\cite[Theorem 1.1]{Porta_Yu_Representability}}]
	Let $F$ be a stack over the étale site of derived \kanal spaces.
	The followings are equivalent:
	\begin{enumerate}
		\item $F$ is a derived \kanal stack;
		\item $F$ is compatible with Postnikov towers, has a global analytic cotangent complex, and its truncation $\trunc(F)$ is an (underived) $k$-analytic stack.
	\end{enumerate}	
\end{thm}

It is non-trivial to construct a global analytic cotangent complex for the $\infty$-functor $\bfMap_S(X,Y)$.
For this purpose, we prove several general results in the context of derived non-archimedean analytic geometry:
derived Tate acyclicity, projection formula, and proper base change.
Since these results deserve independent interest themselves, let us also mention them here:

\begin{thm}[Derived Tate acyclicity, cf.\ \cref{thm:derived_Kiehl}]
	Let $X=(\cX,\cO_X)$ be a derived $k$-affinoid space.
	Let $A$ be the algebra of functions on $X$, which is a simplicial commutative $k$-algebra.
	Then the global section functor gives an equivalence from the stable \infcat of coherent sheaves on $X$ to the stable \infcat of coherent $A$-modules.
\end{thm}

\begin{thm}[Projection formula, cf.\ \cref{thm:projection_formula}]
	Let $p\colon X\to Y$ be a morphism of derived \kanal stacks.
	Assume that $p$ is proper and has finite coherent cohomological dimension (cf.\ \cref{def:coherent_cohomological_dimension}).
	Let $\cF\in\Coh^+(X)$ and $\cG\in\Coh^+(Y)$, where $\Coh^+(-)$ denotes the \infcat of bounded below coherent sheaves.
	Then the canonical map
	\[ \eta_{\cF, \cG} \colon p_*( \cF ) \otimes_{\cO_Y} \cG \longrightarrow p_*( \cF \otimes_{\cO_X} p^*(\cG) ) \]
	is an equivalence.
\end{thm}

\begin{thm} [Proper base change, cf.\ \cref{thm:proper_base_change}] \label{thm-intro:proper_base_change}
	Let
	\[ \begin{tikzcd}
	X' \arrow{r}{u} \arrow{d}{g} & X \arrow{d}{f} \\
	Y' \arrow{r}{v} & Y
	\end{tikzcd} \]
	be a pullback square of derived \kanal stacks.
	Assume that $f$ is proper and has finite coherent cohomological dimension.
	Then for any $\cF \in \Coh^+(X)$, the natural morphism
	\[ v^* f_*(\cF) \to g_* u^*(\cF) \]
	is an equivalence.
\end{thm}

The proof of \cref{thm-intro:proper_base_change} is more involved than its counterpart in the algebraic setting, because pullbacks of derived affinoid spaces are not simply given by algebraic tensor products.
Our proof is achieved in two main steps: we first prove the case where $v$ is a finite map.
In this case, the algebra of functions on $X'$ can be computed via an algebraic tensor product (cf.\ \cref{prop:pullback_finite_map}), so we obtain the result by the algebraic version of the proper base change theorem together with a \v{C}ech descent argument.
Next, we deduce the general statement from this special case via a series of reduction steps.

With the preparations above, we introduce the plus pushforward functor on perfect complexes in \cref{sec:plus_pushforward}.
This enables us to construct the global analytic cotangent complex of the mapping functor in \cref{thm:Map_intro}.
We achieve the proof of the theorem in \cref{sec:mapping_stack}.

The main motivation of our work comes from non-archimedean enumerative geometry.
We refer to the introduction of \cite{Porta_Yu_DNAnG_I} for more details.
In our subsequent works, we will apply the derived mapping stacks to obtain non-archimedean analytic Gromov-Witten invariants.

\bigskip
\paragraph{\textbf{Notations and terminology}}

We refer to \cite{Porta_Yu_DNAnG_I} for the framework of derived non-archimedean analytic geometry, and to \cite{Porta_Yu_Representability} for the theories of modules, square-zero extensions, and cotangent complexes in derived analytic geometry.
We give a quick review of the basic definitions in \cref{sec:review}.

Throughout the paper, $k$ denotes a complete non-archimedean field with nontrivial valuation.
We denote by $\Ank$ the category of rigid \kanal spaces, and by $\dAnk$ the \infcat of derived \kanal spaces.
We denote by $\Afd_k$ the category of rigid $k$-affinoid spaces, and by $\dAfd_k$ the \infcat of derived $k$-affinoid spaces.

For $n\in\Z_{\ge 0}$, we denote by $\bbA^n_k$ the algebraic $n$-dimensional affine space over $k$, by $\bA^n_k$ the analytic $n$-dimensional affine space over $k$, and by $\bD^n_k$ the $n$-dimensional closed unit polydisk over $k$.

We denote by $\cS$ the \infcat of spaces.
An \infsite $(\cC,\tau)$ consists of a small \infcat $\cC$ equipped with a Grothendieck topology $\tau$.
For an \infsite $(\cC,\tau)$ and a presentable infinity category $\cD$, we denote by $\PSh_\cD(\cC)$ the \infcat of $\cD$-valued presheaves on $\cC$, and by $\Sh_\cD(\cC,\tau)$ the \infcat of $\cD$-valued sheaves on the \infsite $(\cC,\tau)$.
We will refer to $\cS$-valued presheaves (resp.\ sheaves) simply as presheaves (resp.\ sheaves), and denote $\PSh(\cC) \coloneqq \PSh_\cS(\cC)$, $\Sh(\cC,\tau) \coloneqq \Sh_\cS(\cC,\tau)$.

We denote by $\DAb$ the derived \infcat of abelian groups, and by $\CRing_k$ the \infcat of simplicial commutative $k$-algebras.
Given an \inftopos $\cX$, we denote by $\CRing_k(\cX)$ the \infcat of sheaves of simplicial commutative $k$-algebras over $\cX$.

A stack over an \infsite $(\cC,\tau)$ is by definition a hypercomplete sheaf with values in $\cS$ over the \infsite (cf.\ \cite[\S 2]{Porta_Yu_Higher_analytic_stacks_2014}).
We denote by $\St(\cC,\tau)$ the \infcat of stacks over $(\cC,\tau)$.

Throughout the paper, we use homological indexing convention, i.e.\ the differential in chain complexes lowers the degree by 1.

\bigskip
\paragraph{\textbf{Acknowledgments}}

We are very grateful to Antoine Chambert-Loir, Maxim Kontsevich, Jacob Lurie, Tony Pantev, Marco Robalo, Carlos Simpson, Bertrand To\"en and Gabriele Vezzosi for valuable discussions.
The authors would also like to thank each other for the joint effort.
Various stages of this research received supports from the Clay Mathematics Institute, Simons Foundation grant number 347070, and from the Ky Fan and Yu-Fen Fan Membership Fund and the S.-S.\ Chern Endowment Fund of the Institute for Advanced Study.

\section{Review of basic definitions} \label{sec:review}

In this section, we give a quick review of the basic definitions in derived non-archimedean analytic geometry following \cite{Porta_Yu_DNAnG_I,Porta_Yu_Representability}.
We refer to the introduction of \cite{Porta_Yu_DNAnG_I} for an heuristic explanation of the ideas behind the abstract definitions.

The notion of derived non-archimedean analytic space is based on the theory of pregeometry and structured topos introduced by Lurie \cite{DAG-V}.

\begin{defin}[{\cite[3.1.1]{DAG-V}}]
	A \emph{pregeometry} is an \infcat $\cT$ equipped with a class of \emph{admissible} morphisms and a Grothendieck topology generated by admissible morphisms, satisfying the following conditions:
	\begin{enumerate}[(i)]
		\item The \infcat $\cT$ admits finite products.
		\item The pullback of an admissible morphism along any morphism exists.
		\item The class of admissible morphisms is closed under composition, pullback and retract.
		Moreover, for morphisms $f$ and $g$, if $g$ and $g\circ f$ are admissible, then $f$ is admissible.
	\end{enumerate}
\end{defin}

\begin{defin}[{\cite[3.1.4]{DAG-V}}] \label{def:structure}
	Let $\cT$ be a pregeometry, and let $\cX$ be an \inftopos.
	A \emph{$\cT$-structure} on $\cX$ is a functor $\cO\colon\cT\to\cX$ with the following properties:
	\begin{enumerate}[(i)]
		\item The functor $\cO$ preserves finite products.
		\item Given any pullback diagram
		\[
		\begin{tikzcd}
		U' \arrow{r} \arrow{d} & U \arrow{d}{f} \\
		X' \arrow{r} & X
		\end{tikzcd}
		\]
		in $\cT$, where $f$ is admissible, the induced diagram
		\[
		\begin{tikzcd}
		\cO(U') \arrow{r} \arrow{d} & \cO(U) \arrow{d} \\
		\cO(X') \arrow{r} & \cO(X)
		\end{tikzcd}
		\]
		is a pullback square in $\cX$.
		\item Let $\{U_\alpha\to X\}$ be a covering in $\cT$ consisting of admissible morphisms.
		Then the induced map
		\[\coprod_\alpha\cO(U_\alpha)\to\cO(X)\]
		is an effective epimorphism in $\cX$.
	\end{enumerate}
	A morphism of $\cT$-structures $\cO\to\cO'$ on $\cX$ is \emph{local} if for every admissible morphism $U\to X$ in $\cT$, the resulting diagram
	\[ \begin{tikzcd}
	\cO(U) \arrow{r} \arrow{d} & \cO'(U) \arrow{d} \\
	\cO(X) \arrow{r} & \cO'(X)
	\end{tikzcd} \]
	is a pullback square in $\cX$.
	We denote by $\Strloc_\cT(\cX)$ the \infcat of $\cT$-structures on $\cX$ with local morphisms.
	
	A \emph{$\cT$-structured \inftopos} $X$ is a pair $(\cX,\cO_X)$ consisting of an \inftopos $\cX$ and a $\cT$-structure $\cO_X$ on $\cX$.
	We denote by $\RTop(\cT)$ the \infcat of $\cT$-structured \inftopoi (cf.\ \cite[Definition 1.4.8]{DAG-V}).
	Note that a 1-morphism $f\colon (\cX, \cO_X) \to (\cY, \cO_Y)$ in $\RTop(\cT)$ consists of a geometric morphism of \inftopoi $f_*\colon\cX\rightleftarrows\cY\colon f\inv$ together with a local morphism of $\cT$-structures $f^\sharp \colon f\inv \cO_Y \to \cO_X$.
\end{defin}

We fix $k$ a complete non-archimedean field with nontrivial valuation.
Two important pregeometries in the theory of derived non-archimedean analytic geometry are $\cTank$ and $\cTdisck$, which we recall below:

The pregeometry $\cTank$ is defined as follows:
\begin{enumerate}[(i)]
	\item The underlying category of $\cTank$ is the category of smooth rigid $k$-analytic spaces;
	\item A morphism in $\cTank$ is admissible if and only if it is étale;
	\item The topology on $\cTank$ is the étale topology.
\end{enumerate}

The pregeometry $\cTdisck$ is defined as follows:
\begin{enumerate}[(i)]
	\item The underlying category of $\cTdisck$ is the full subcategory of the category of $k$-schemes spanned by affine spaces $\mathbb A^n_k$;
	\item A morphism in $\cTdisck$ is admissible if and only if it is an isomorphism;
	\item The topology on $\cTdisck$ is the trivial topology, i.e.\ a collection of admissible morphisms is a covering if and only if it is nonempty.
\end{enumerate}

We have a natural functor $\cTdisck \to \cTank$ induced by analytification.
Composing with this functor, we obtain an ``algebraization'' functor
\[
(-)^\mathrm{alg} \colon \Strloc_{\cTank}(\cX) \to \Strloc_{\cTdisck}(\cX).
\]
As in \cite[Example 3.1.6, Remark 4.1.2]{DAG-V}, we have an equivalence induced by evaluation on the affine line
\[\Strloc_{\cTdisck}(\cX) \xrightarrow{\ \sim\ } \Sh_{\CRing_k}(\cX),\]
where $\Sh_{\CRing_k}(\cX)$ denotes the \infcat of sheaves on $\cX$ with values in the \infcat of simplicial commutative $k$-rings.

\begin{defin}[cf.\ {\cite[Definition 2.5]{Porta_Yu_DNAnG_I}}] \label{def:derived_analytic_space}
	A \emph{derived \kanal space} $X$ is a $\cTank$-structured \inftopos $(\cX,\cO_X)$ such that $\cX$ is hypercomplete and there exists an effective epimorphism from $\coprod_i U_i$ to a final object of $\cX$ satisfying the following conditions, for every index $i$:
	\begin{enumerate}[(i)]
		\item The pair $(\cX_{/U_i}, \pi_0(\cO\alg_X |_{U_i}))$ is equivalent to the ringed \inftopos associated to the étale site of a rigid \kanal space $X_i$.
		\item For each $j\ge 0$, $\pi_j(\cO\alg_X |_{U_i})$ is a coherent sheaf of $\pi_0(\cO\alg_X |_{U_i})$-modules on $X_i$.
	\end{enumerate}
	We denote by $\dAnk$ the full subcategory of $\RTop(\cTank)$ spanned by derived \kanal spaces.
\end{defin}

\begin{defin}[cf.\ {\cite[Definition 7.3]{Porta_Yu_DNAnG_I}}]
	A \emph{derived $k$-affinoid space} is a derived $k$-analytic space $(\cX, \cO_X)$ whose truncation $(\cX, \pi_0(\cO\alg_X))$ is a $k$-affinoid space.
	We denote by $\dAfdk$ the $\infty$-category of derived $k$-affinoid spaces.
\end{defin}

We refer to \cite[\S 2]{Porta_Yu_Higher_analytic_stacks_2014} for the notions of geometric context and geometric stack with respect to a given geometric context.
Recall that a geometric context $(\cC,\tau,\bP)$ consists of a small \infcat $\cC$ equipped with a Grothendieck topology $\tau$ and a class $\bP$ of morphisms in $\cC$, satisfying a short list of axioms.

A stack over an \infsite $(\cC,\tau)$ is by definition a hypercomplete sheaf with values in spaces over the \infsite.
We denote by $\St(\cC,\tau)$ the \infcat of stacks over $(\cC,\tau)$.

Given a geometric context $(\cC,\tau,\bP)$ and an integer $n\ge -1$, the notion of $n$-geometric stack is defined by induction on the geometric level $n$.
We refer to \cite[\S 2.3]{Porta_Yu_Higher_analytic_stacks_2014} for the details.
Let us simply recall that a $(-1)$-geometric stack is by definition a representable stack.

\begin{defin}
	Let $(\dAfd_k, \tauet, \bPsm)$ denote the geometric context consisting of the \infcat of derived $k$-affinoid spaces, the étale topology (cf.\ \cite[Definition 5.1]{Porta_Yu_DNAnG_I}) and the class of smooth morphisms (cf.\ \cite[Definition 5.45]{Porta_Yu_Representability}).
	A \emph{derived $k$-analytic stack} is an $n$-geometric stack with respect to the geometric context $(\dAfd_k, \tauet, \bPsm)$ for some $n$.
\end{defin}

Note that the notions of étaleness and smoothness of morphisms between derived \kanal spaces is local both on the source and on the target, so they extend naturally to all morphisms between derived \kanal stacks.

\begin{defin} \label{def:coherent_sheaf}
	Let $X$ be a derived \kanal stack.
	Let $(\dAfd_{k/X})_{\bPsm}$ denote the full subcategory of the overcategory $\St(\dAfd_k,\tauet)_{/X}$ spanned by smooth morphisms from derived $k$-affinoid spaces to $X$.
	Consider the functor
	\begin{align*}
	\cO_X\alg \colon (\dAfd_{k/X})_{\bPsm} &\longrightarrow \CRing_k\\
	(U \to X) &\longmapsto \Gamma(\cO_U\alg) ,
	\end{align*}
	where $\CRing_k$ is the $\infty$-category of simplicial commutative $k$-rings.
	
	Let $((\dAfd_{k/X})_{\bPsm},\tauet)$ denote the category $(\dAfd_{k/X})_{\bPsm}$ equipped with the étale topology.
	Let $\DAb$ denote the derived \infcat of abelian groups.
	Let $\Sh_{\DAb}((\dAfd_{k/X})_{\bPsm}, \tauet)^\wedge$ denote the \infcat of hypercomplete $\DAb$-valued sheaves on the site $((\dAfd_{k/X})_{\bPsm},\tauet)$.
	Then $\cO\alg_X$ is an element there.
	We define the stable \infcat $\cO_X \Mod$ of \emph{$\cO_X$-modules} to be the stable $\infty$-category of $\cO_X\alg$-modules in $\Sh_{\DAb}((\dAfd_{k/X})_{\bPsm}, \tauet)^\wedge$.
	The stable \infcat $\cO_X\Mod$ is naturally equipped with a t-structure (cf.\ \cite[2.1.3]{DAG-VIII}).
	We define the stable \infcat $\Coh(X)$ of \emph{coherent sheaves on $X$} to be the full subcategory of $\cO_X\Mod$ spanned by $\cF\in\cO_X\Mod$ such that $\pi_i(\cF)$ is a coherent sheaf of $\pi_0(\cO_X\alg)$-modules for every $i$.
	For every $n \in \mathbb Z$, we set
	\begin{gather*}
	\Coh^{\ge n}(X) \coloneqq \Coh(X) \cap \cO_X \Mod^{\ge n}, \quad \Coh^{\le n}(X) \coloneqq \Coh(X) \cap \cO_X \Mod^{\le n} , \\
	\Coh^+(X) \coloneqq \Coh(X) \cap \cO_X \Mod^+, \quad \Coh^-(X) \coloneqq \Coh(X) \cap \cO_X \Mod^- ,\\
	\Cohb(X)\coloneqq\Coh^+(X)\cap\Coh^-(X).
	\end{gather*}
\end{defin}

\section{Derived Tate acyclicity} \label{sec:acyclicity}

In this section, we extend Tate acyclicity theorem to derived affinoid spaces.

\begin{thm} \label{thm:derived_Tate_acyclicity}
	Let $X = (\cX, \cO_X)$ be a derived $k$-affinoid space.
	Let $A \coloneqq \Gamma(\cO_X\alg)$, which is an $\mathbb E_\infty$-ring spectrum over $k$.
	Then $A$ is connective and simplicial.
	Moreover, the global section functor
	\[ \Gamma \colon \cO_X \Mod \longrightarrow A \Mod \]
	restricts to a $t$-exact functor
	\[ \Coh(X) \xrightarrow{\ \sim\ }  \Coh(A). \]
\end{thm}

\begin{proof}
	Let $\cF \in \cO_X \Mod$.
	Since $X$ is a derived \kanal space, the $\infty$-topos $\cX$ is by definition hypercomplete.
	As a consequence, \cite[Proposition 5.3]{Porta_Yu_Higher_analytic_stacks_2014} allows us to write
	\[ \cF \simeq \lim_{n \ge 0} \tau_{\le n} \cF \in \Sh_{\DAb}(\cX), \]
	where $\DAb$ denotes the derived \infcat of abelian groups.
	Since the global section functor $\Gamma$ commutes with limits, we obtain
	\[ \Gamma(\cF) \simeq \lim_{n \ge 0} \Gamma(\tau_{\le n} \cF) \in \DAb . \]
	The dual of \cite[1.2.2.14]{Lurie_Higher_algebra} yields a spectral sequence
	\begin{equation} \label{eq:spectral_sequence}
		E_2^{i,j} \coloneqq \pi_i \left( \Gamma(\pi_j(\cF)) \right) \Longrightarrow \pi_{i+j}\left( \Gamma(\cF) \right) .
	\end{equation}

	Assume now that $\cF \in \Coh(X)$.
	Then the sheaves $\pi_j(\cF)$ are coherent sheaves over the underived \kanal space $\trunc(X) = (\cX, \pi_0(\cO\alg_X))$.
	In particular, the Tate acyclicity theorem implies that
	\[ \pi_i \left( \Gamma(\cF) \right) = 0 \]
	for every $i \ne 0$.
	This guarantees at the same time the convergence of the spectral sequence \eqref{eq:spectral_sequence} and the degeneration at the second page.

	Applying this to $\cF = \cO_X\alg$, we deduce that $A = \Gamma(\cO_X\alg)$ is connective, and therefore simplicial.
	In particular, we can use \cite[Proposition 2.1.3]{DAG-VIII} in order to endow $A\Mod$ with a $t$-structure.
	Furthermore, for a general $\cF \in \Coh(X)$, we can deduce from the Tate acyclicity theorem that $\Gamma(\pi_j(\cF))$ is finitely generated as $\pi_0(A)$-module.
	In turn, this implies that $\Gamma(\cF) \in \Coh(A)$.
	In other words, $\Gamma$ restricts to a functor of stable $\infty$-categories
	\[ \Gamma \colon \Coh(X) \to \Coh(A) . \]
	We are left to prove that this functor is $t$-exact.
	In order to see this, it is enough to check that for any $\cF \in \Cohh(X)$, we have $\Gamma(\cF) \in \Cohh(A)$.
	Once again, this follows from the spectral sequence \eqref{eq:spectral_sequence} and the Tate acyclicity theorem.
\end{proof}

\begin{lem} \label{lem:global_section_etale_is_flat}
	Let $X = (\cX, \cO_X)$ and $Y = (\cY, \cO_Y)$ be two derived $k$-affinoid spaces.
	Let $A \coloneqq \Gamma(\cO_X\alg)$ and $B \coloneqq \Gamma(\cO_Y\alg)$.
	They are simplicial commutative $k$-algebras by \cref{thm:derived_Tate_acyclicity}.
	If $f \colon Y \to X$ is an étale morphism, then the induced morphism $\varphi \colon A \to B$ is flat.
\end{lem}

\begin{proof}
	Let us prove first that $\pi_0(\varphi) \colon \pi_0(A) \to \pi_0(B)$ is flat.
	We can identify the rings $\pi_0(A)$ and $\pi_0(B)$ respectively with the rings of global sections of the $k$-affinoid spaces $\trunc(X)$ and $\trunc(Y)$.
	Since $f$ is étale, its truncation $\trunc(f) \colon \trunc(Y) \to \trunc(X)$ is étale as well.
	In particular, it is flat.
	Therefore, $\pi_0(\varphi) \colon \pi_0(A) \to \pi_0(B)$ is also flat.
	
	Let us now prove that $\varphi$ is strong.
	For this, we first observe that \cref{thm:derived_Tate_acyclicity} implies that $\pi_i(A) \simeq \Gamma(\pi_i(\cO\alg_X))$.
	Since $f$ is étale, the canonical morphism $f\inv \cO_X \to \cO_Y$ is an isomorphism.
	As $f\inv$ is $t$-exact, \cref{thm:derived_Tate_acyclicity} provides the following chain of equivalences:
	\[ \Gamma(f\inv(\pi_i(\cO\alg_X))) \simeq \Gamma(\pi_i(f\inv \cO\alg_X )) \simeq \Gamma(\pi_i(\cO\alg_Y)) \simeq \pi_i(B) . \]
	Since $\pi_i(\cO\alg_X)$ is a coherent sheaf on $\trunc(X)$, fpqc descent \cite{Conrad_Descent_for_coherent_2003} and Kiehl's theorem for coherent modules \cite[\S 9.4.3]{Bosch_Non-archimedean_1984} imply together that
	\[ \pi_i(A) \otimes_{\pi_0(A)} \pi_0(B) \simeq \Gamma(f\inv(\pi_i(\cO\alg_X))) \simeq \pi_i(B) . \]
	Therefore, $\varphi$ is strong, and the proof is achieved.
\end{proof}

\begin{thm} \label{thm:derived_Kiehl}
	Let $X = (\cX, \cO_X)$ be a derived $k$-affinoid space.
	Let $A \coloneqq \Gamma(\cO_X\alg)$, which is a simplicial commutative $k$-algebra by \cref{thm:derived_Tate_acyclicity}.
	Then the global section functor 
	\[ \Gamma \colon \cO_X \Mod \longrightarrow A \Mod \]
	restricts to a t-exact equivalence of stable \infcats:
	\[ \Coh(X) \xrightarrow{\ \sim\ }  \Coh(A).\]
\end{thm}

\begin{proof}
	We already know from \cref{thm:derived_Tate_acyclicity} that $\Gamma$ restricts to a $t$-exact functor of stable \infcats
	\[ \Coh(X) \longrightarrow \Coh(A) . \]
	Let us construct an inverse functor.
	Consider the $\infty$-functor
	\[ (\dAfd_{k/X})_{\bPsm}\op \times \Coh(A) \to \DAb \]
	defined by
	\[ (U, M) \mapsto \cO_X(U) \otimes_A M . \]
	This $\infty$-functor induces
	\[ c \colon \Coh(A) \to \PSh_{\DAb}((\dAfd_{k/X})_{\bPsm}) . \]
	We claim that $c$ factors through $\Sh_{\DAb}((\dAfd_{k/X})_{\bPsm},\tauet)^\wedge$.
	In other words, we have to prove that for every $M \in \Coh(A)$, the presheaf $c(M)$ defined by
	\[ c(M)(U) \coloneqq \cO_X(U) \otimes_A M \]
	is a hypercomplete sheaf.
	
	Observe that when $M$ is discrete, \cref{lem:global_section_etale_is_flat} implies that $\cO_X(U) \otimes_A M$ is again discrete, and hence it is equivalent to
	\[ \Tor^{\pi_0(A)}_0( \pi_0(\cO_X(U)), M ) . \]
	Therefore, Kiehl's theorem for coherent modules \cite[\S 9.4.3]{Bosch_Non-archimedean_1984} and the fpqc descent theorem by Conrad \cite{Conrad_Descent_for_coherent_2003} imply that $c(M)$ is hypercomplete.
	Finally, let $M \in \Coh(A)$ and write
	\[ M \simeq \colim_n \tau_{\ge n} M . \]
	Since $c$ is $t$-exact, we have
	\[ c(M) \simeq \colim_n c(\tau_{\ge n} M) . \]
	Let $U^\bullet$ be a hypercovering of $U$ and consider the map
	\[ c(M)(U) \to \lim_{\mathbf \Delta} c(M)(U^\bullet) . \]
	Using \cref{lem:global_section_etale_is_flat} we see that the complexes $E^{n\bullet}_1$ forming the $E_1$-page of the spectral sequence associated to the limit above are the complexes associated (via the Dold-Kan correspondence) to the cosimplicial object $c(\rH^n(M))(U^\bullet)$.
	Therefore, by \cite[\S 9.4.3]{Bosch_Non-archimedean_1984} and \cite{Conrad_Descent_for_coherent_2003} again, we see that the spectral sequence degenerates at the $E_2$-page, yielding
	\[ c(\rH^n(M))(U) \simeq \rH^n\left( \lim_{\mathbf \Delta} c(M)(U^\bullet) \right) . \]
	As \cref{lem:global_section_etale_is_flat} implies that
	\[ c(\rH^n(M))(U) \simeq \rH^n(c(M)(U)) , \]
	we conclude that $c(M)$ is hypercomplete.
	
	Finally, we observe that $c$ can be promoted to an $\infty$-functor
	\[ \tilde{c} \colon \Coh(A) \to \cO_X \Mod . \]
	\Cref{lem:global_section_etale_is_flat} implies that $\rH^i( \tilde{c}(M) ) \simeq \tilde{c}( \rH^i(M) )$.
	Therefore, the classical Kiehl's theorem implies that $\tilde{c}$ factors through $\Coh(X)$.
	Using the $t$-exactness of $\tilde{c}$ and of $\Gamma$, we can now prove that they form an equivalence of stable $\infty$-categories.
\end{proof}

\begin{cor} \label{cor:geometric_realization_almost_perfect}
	Let $X$ be a derived \kanal stack.
	Then $\Coh^{\ge 0}(X)$ admits geometric realizations and the inclusion $\Coh^{\ge 0}(X) \hookrightarrow \cO_X \Mod$ preserves them.
\end{cor}

\begin{proof}
	By definition, $\cO_X\Mod$ satisfies étale descent.
	Since the étale topology and the smooth topology are equivalent, $\cO_X\Mod$ also satisfies smooth descent.
	By \cite{Conrad_Descent_for_coherent_2003}, $\Coh^{\ge 0}(X)$ is closed under smooth descent in $\cO_X\Mod$.
	So $\Coh^{\ge 0}(X)$ satisfies smooth descent as well.
	Therefore, it suffices to prove the corollary locally.
	In other words, we can assume $X$ to be derived affinoid.
	In this case, the existence of geometric realizations in $\Coh^{\ge 0}(X)$ is a consequence of \cref{thm:derived_Kiehl} and \cite[7.2.4.11(4)]{Lurie_Higher_algebra}.
	
	Consider now the site $\big((\dAfd_{k/X})_{\bP_\mathrm{sm}},
	\tauet\big)$ as in \cref{def:coherent_sheaf}.
	We have a conservative functor
	\[ \Coh^{\ge 0}(X) \to \PSh_{\DAb}( (\dAfd_{/X})_{\bP_\mathrm{sm}} ) . \]
	Using again \cref{thm:derived_Kiehl} and \cite[7.2.4.11(4)]{Lurie_Higher_algebra}, we see that this functor commutes with geometric realizations.
	In particular, if $\cF^\bullet$ is a simplicial object in $\Coh^{\ge 0}(X)$ and let $\cF$ denote its geometric realization in $\PSh_{\DAb}( (\dAfd_{/X})_{\bP_\mathrm{sm}} )$, then $\cF$ is a sheaf for the étale topology.
	This implies that $\cF$ is a also geometric realization of $\cF^\bullet$ in $\cO_X \Mod$, thus completing the proof.
\end{proof}

\begin{cor} \label{cor:derived_affinoid_presentation_of_coherent}
	Let $X$ be a derived $k$-affinoid space.
	Let $\cF \in \Coh^+(X)$.
	Then there exists a simplicial object $\cP^\bullet \in \Fun(\mathbf \Delta\op, \Coh^+(X))$ such that each $\cP^n$ is free of finite rank and
	\[ |\cP^\bullet| \simeq \cF \]
	in $\Coh^+(X)$.
\end{cor}

\begin{proof}
	This is a direct consequence of \cref{thm:derived_Kiehl} and \cite[7.2.4.11(5)]{Lurie_Higher_algebra}.
\end{proof}

\section{Pullback along finite maps} \label{sec:pullback_finite_maps}

The goal of this section is to prove the following \cref{prop:pullback_finite_map}.

\begin{defin}
	A morphism $f \colon X \to Y$ of derived \kanal spaces is said to be \emph{finite} if its truncation $\trunc(f)$ is finite.
\end{defin}

\begin{prop} \label{prop:pullback_finite_map}
	Let
	\begin{equation} \label{eq:pullback_finite_map}
	\begin{tikzcd}
	X' \arrow{r}{u} \arrow{d}{g} & X \arrow{d}{f} \\
	Y' \arrow{r}{v} & Y
	\end{tikzcd}
	\end{equation}
	be a pullback square of derived $k$-affinoid spaces.
	Set $A \coloneqq \Gamma(\cO\alg_Y)$, $A' \coloneqq \Gamma(\cO\alg_{Y'})$, $B \coloneqq \Gamma(\cO\alg_X)$, and $B' \coloneqq \Gamma(\cO\alg_{X'})$.
	If $v$ is finite, then the canonical map
	\[ \eta \colon A' \otimes_A B \longrightarrow B' \]
	is an equivalence.
\end{prop}

Let us first consider a special case:

\begin{lem} \label{lem:global_section_pullback_double_closed_immersions}
	Assume furthermore that in the square \eqref{eq:pullback_finite_map} the maps $f$ and $v$ are closed immersions.
	Then the square
	\[ \begin{tikzcd}
	\cO\alg_Y \arrow{r} \arrow{d} & v_* \cO\alg_{Y'} \arrow{d} \\
	f_* \cO\alg_X \arrow{r} & f_* u_* \cO\alg_{X'}
	\end{tikzcd} \]
	is a pushout square in $\CRing_k(\cY)$.
	Moreover, it remains a pushout square after taking global sections.
\end{lem}
\begin{proof}
	Denote by $\cX$ (resp.\ $\cX'$, $\cY$, $\cY'$) the underyling $\infty$-topos of $X$ (resp.\ $X'$, $Y$, $Y'$).
	By \cite[Proposition 3.5]{Porta_Yu_DNAnG_I}, $f$ and $v$ induce closed immersions of $\infty$-topoi
	\[ f_* \colon \cX \leftrightarrows \cY \colon f\inv, \qquad v_* \colon \cY' \leftrightarrows \cY \colon v\inv . \]
	Since $\cY$ is hypercomplete, we can check the first statement on the geometric points of $\cY$.
	Let $y_* \colon \cS \leftrightarrows \cY \colon y\inv$ be a geometric point of $\cY$.
	Notice that if $y_*$ does not factor through $v_*$, then $y\inv v_* \cO\alg_{Y'} \simeq v_* g_* \cO\alg_{X'} \simeq 0$.
	Similarly, if $y_*$ does not factor through $f_*$, then $y\inv f_* \cO\alg_X \simeq f_* u_* \cO\alg_{X'} \simeq 0$.
	In both cases, the statement is verified.
	We are left to consider the case where $y$ factors through both $v_*$ and $f_*$.
	Using the proper base change for closed morphism of $\infty$-topoi (see \cite[7.3.2.13]{HTT}) and the fully faithfulness of $f_*$ and $v_*$, we reduce ourselves to check that the square
	\[ \begin{tikzcd}
	u\inv f\inv \cO\alg_Y \arrow{r} \arrow{d} & v\inv \cO\alg_{Y'} \arrow{d} \\
	f\inv \cO\alg_X \arrow{r} & \cO\alg_{X'}
	\end{tikzcd} \]
	is a pushout square in $\CRing_k(\cX')$.
	This is true in virtue of \cite[Proposition 3.17]{Porta_Yu_DNAnG_I}.
	
	Observe furthermore that both $f_* \cO\alg_X$ and $v_* \cO\alg_{Y'}$ are in particular coherent sheaves over $Y$.
	So the second statement follows directly from the equivalence provided by \cref{thm:derived_Kiehl}.
\end{proof}

In order to deal with the general case, we will prove that for every derived $k$-affinoid space $X$, the canonical projection $X \times \bD^n_k \to X$ induces a flat map between the global sections.
Let us start with a couple of preliminaries.

Let $V$ be a derived $k$-affinoid space.
Let $\cF \in \Coh^{\ge 1}(V)$, $d \colon \anL_V \to \cF$ an analytic derivation, and $V_d[\cF]$ the associated square-zero extension (cf.\ \cite[\S 5]{Porta_Yu_Representability}).
By \cite[Theorem 6.5]{Porta_Yu_Representability}, we have a pushout square
\[ \begin{tikzcd}
V[\cF] \arrow{r}{\eta_d} \arrow{d}{\eta_0} & V \arrow{d} \\
V \arrow{r} & V_d[\cF]
\end{tikzcd} \]
in $\dAnk$, where $\eta_0$ corresponds to the zero derivation and $\eta_d$ corresponds to the derivation $d$.

\begin{lem} \label{lem:product_preserves_infinitesimal_pushout}
	For any derived \kanal space $X$, the square
	\[ \begin{tikzcd}
	V[\cF] \times X \arrow{r}{\eta_d \times \mathrm{id}_X} \arrow{d}{\eta_0 \times \mathrm{id}_X} & V \times X \arrow{d} \\
	V \times X \arrow{r} & V_d[\cF] \times X
	\end{tikzcd} \]
	is a pushout square in $\dAnk$.
\end{lem}

\begin{proof}
	By \cite[Lemma 7.7]{Porta_Yu_Representability}, every derived \kanal space is infinitesimally cartesian.
	So it is enough to prove that if $F$ is an infinitesimally cartesian sheaf on $(\dAfdk, \tauet)$, then the canonical map
	\[ F(V_d[\cF] \times X) \to F(V \times X) \times_{F(V[\cF] \times X)} F(V \times X) \]
	is an equivalence.
	We can rewrite the above map as
	\begin{multline*}
	\Map(V_d[\cF], \bfMap(X, F)) \to \\
	\Map(V, \bfMap(X, F)) \times_{\Map(V[\cF], \bfMap(X, F))} \Map(V, \bfMap(X, F)) .
	\end{multline*}	\Cref{lem:map_preserves_inf_cartesian} implies that $\bfMap(X, F)$ is again infinitesimally cartesian.
	Therefore the lemma holds.
\end{proof}

\begin{defin}
	Let $X$ be a derived $k$-affinoid space and set $A \coloneqq \Gamma(\cO\alg_X)$.
	We denote $A \langle T_1, \ldots, T_n \rangle\coloneqq\Gamma\big(\cO\alg_{X \times \bD^n_k}\big)$, and call it the \emph{relative Tate algebra} of dimension $n$ over $X$.
\end{defin}

\begin{lem} \label{lem:relative_disk_flat_I}
	Let $X$ be a derived $k$-affinoid space.
	For every $n,m \ge 0$, we have a canonical equivalence
	\[ \mathrm t_{\le m}(X) \times \bD^n_k \simeq \mathrm t_{\le m}( X \times \bD^n_k ) . \]
\end{lem}

\begin{proof}
	Using \cite[Lemma 5.46]{Porta_Yu_Representability} we can find a closed embedding $j \colon X \hookrightarrow \mathbf D^l_k$ for some $l \ge 0$.
	Then we have canonical equivalences
	\[ (\mathrm t_{\le m} X) \times \mathbf D^n_k \simeq ( \mathrm t_{\le m} X ) \times_{\mathbf D^l_k} \mathbf D^{l+n}_k , \qquad X \times \mathbf D^n_k \simeq X \times_{\mathbf D^l_k} \mathbf D^{l+n}_k . \]
	Denote by $p \colon \bD^{l+n}_k \to \bD^n_k$ the canonical projection.
	Combining \cite[Proposition 3.17]{Porta_Yu_DNAnG_I} with the above equivalences, we can express the structure sheaf of $X \times \bD^n_k$ as the tensor product
	\[ p\inv j_* \cO_X \otimes_{p\inv \cO_{\bD^l_k}} \cO_{\bD^{l+n}_k} . \]
	Notice now that both $j_*$ and $p\inv$ commute with the truncation functors.
	Therefore, the lemma follows from the flatness of $p\inv \cO_{\bD^n_k} \to \bD^{l+n}_k$.
\end{proof}

\begin{lem} \label{lem:pullback_split_square_zero_extension}
	Let $X, Y$ be derived \kanal spaces and let $ \cF \in \Coh^{\ge 1}(X)$.
	The canonical projection $p \colon X \times Y \to X$ induces an equivalence
	\[ X[\cF] \times Y \simeq (X \times Y)[p^*(\cF)] . \]
\end{lem}

\begin{proof}
	Let $f_0 \colon X[\cF] \to X$ be the map corresponding to the zero derivation $\anL_X \to \cF$.
	Consider the pullback diagram
	\begin{equation} \label{eq:pullback_split_square_zero_extension}
		\begin{tikzcd}
			X[\cF] \times Y \arrow{r} \arrow{d} & X[\cF] \arrow{d}{f_0} \\
			X \times Y \arrow{r}{p} \arrow{d} & X \arrow{d} \\
			Y \arrow{r} & \Sp(k) .
		\end{tikzcd}
	\end{equation}
	Notice that $p$ induces a canonical morphism of $\cTank$-structures on the $\infty$-topos underlying $X \times Y$
	\[ p\inv( \cO_X \oplus \cF ) \longrightarrow \cO_{X \times Y} \oplus p^*(\cF) , \]
	corresponding to a morphism of derived \kanal spaces $(X \times Y)[p^*(\cF)] \to X[\cF]$.
	In turn this induces a well-defined morphism $(X \times Y)[p^*(\cF)] \to X[\cF] \times Y$.
	We claim that this morphism is an equivalence.
	
	Since $\cF \in \Coh^{\ge 1}(X)$, we see that $f_0$ is a closed immersion.
	We can therefore use \cite[Proposition 6.2]{Porta_Yu_DNAnG_I} to describe the upper pullback in \eqref{eq:pullback_split_square_zero_extension} as follows: its underlying $\infty$-topos coincides with the underlying $\infty$-topos of $X \times Y$, while the underlying algebra of its structure sheaf is the pushout
	\[ \begin{tikzcd}
		p\inv \cO_X\alg \arrow{r} \arrow{d} & p\inv(\cO_X\alg \oplus \cF) \arrow{d} \\
		\cO_{X \times Y}\alg \arrow{r} & \cO_{X[\cF] \times Y}\alg .
	\end{tikzcd} \]
	As this is a pushout in the $\infty$-category of sheaves of simplicial commutative algebras over $X \times Y$, we can canonically identify the above pushout with the split square-zero extension $\cO_{X \times Y}\alg \oplus p^*(\cF)$.
	Now the conclusion follows because $(-)\alg$ is conservative.
\end{proof}

\begin{prop} \label{prop:relative_disk_flat_II}
	Let $X$ be a derived $k$-affinoid space and set $A \coloneqq \Gamma(\cO\alg_X)$.
	The canonical map
	\[ A \to A \langle T_1, \ldots, T_n \rangle \]
	is flat.
\end{prop}

\begin{proof}
	We proceed by induction on the Postnikov tower of $X$.
	\Cref{lem:relative_disk_flat_I} implies that if $X$ is discrete then the same goes for $X \times \bD^n_k$.
	Using \cref{thm:derived_Kiehl}, we deduce that in this case $A \langle T_1, \ldots, T_n \rangle$ is discrete as well.
	In this case, both $A$ and $A\langle T_1, \ldots, T_n \rangle$ are noetherian.
	We can therefore check flatness on the formal completions at the maximal ideals of $A \langle T_1, \ldots, T_n \rangle$.
	As we can identify such formal completions with formal power series rings, the proposition holds in this case.
	
	Assume now that the statement has already been proven when the structure sheaf of $X$ is $m$-truncated.
	Set $\cF \coloneqq \pi_{m+1} \cO\alg_X[m+2]$.
	Combining \cref{lem:product_preserves_infinitesimal_pushout} and \cite[Corollary 5.42]{Porta_Yu_Representability}, we deduce that the square
	\[ \begin{tikzcd}
		(\mathrm t_{\le m} X[\cF]) \times \mathbf D^n_k \arrow{r} \arrow{d} & (\mathrm t_{\le m} X) \times \mathbf D^n_k \arrow{d} \\
		(\mathrm t_{\le m} X) \times \mathbf D^n_k \arrow{r} & (\mathrm t_{\le m+1} X) \times \mathbf D^n_k
	\end{tikzcd} \]
	is a pushout in $\dAnk$.
	Denote by $p_m \colon \mathrm t_{\le m} X \times \bD^n_k \to \mathrm t_{\le m} X$ the canonical projection.
	Furthermore, let $\cX$ be the underlying $\infty$-topos of $X$ and $\cY$ the underlying $\infty$-topos of $X \times \bD^n_k$.
	The morphism $p_m$ induces a geometric morphism of $\infty$-topoi (independent of $m$)
	\[ p\inv \colon \cX \leftrightarrows \cY \colon p_* . \]
	Combining Lemmas \ref{lem:relative_disk_flat_I} and \ref{lem:pullback_split_square_zero_extension}, we obtain an equivalence
	\begin{multline*}
	\pi_{m+1}(\cO\alg_{X \times \bD^n_k})[m+2] \simeq \pi_{m+1}( \tau_{\le m+1} \cO\alg_{X \times \bD^n_k} )[m+2] \\
	\simeq p\inv(\cF) \otimes_{p\inv( \tau_{\le m} \cO\alg_X )} \tau_{\le m} \cO\alg_{X \times \bD^n_k} \simeq p_m^*( \cF ) .
	\end{multline*}
	Passing to the global sections and using \cref{thm:derived_Kiehl}, we obtain
	\begin{align*}
		\pi_{m+1}( A \langle T_1, \ldots, T_n \rangle ) & \simeq \pi_{m+1}(A) \otimes_{\tau_{\le m} A} \tau_{\le m} ( A \langle T_1, \ldots, T_n \rangle ) \\
		& \simeq \pi_{m+1}(A) \otimes_{\pi_0(A)} \pi_0(A) \otimes_{\tau_{\le m} A} \tau_{\le m}( A \langle T_1, \ldots, T_n \rangle ) .
	\end{align*}
	The induction hypothesis guarantees that $\tau_{\le m}( A \langle T_1, \ldots, T_n \rangle)$ is flat over $\tau_{\le m}(A)$.
	Therefore we have
	\[ \pi_0(A) \otimes_{\tau_{\le m} A} \tau_{\le m}(A \langle T_1, \ldots, T_n \rangle) \simeq \pi_0(A \langle T_1, \ldots, T_n \rangle ) . \]
	Combining this with the above expression for $\pi_{m+1}( A \langle T_1, \ldots, T_n \rangle )$ we obtain an isomorphism
	\[ \pi_{m+1}( A \langle T_1, \ldots, T_n \rangle ) \simeq \pi_{m+1}(A) \otimes_{\pi_0(A)} \pi_0(A \langle T_1, \ldots, T_n \rangle ) . \]
	Together with the induction hypothesis, this guarantees that the map
	\[ \tau_{\le m+1} A \longrightarrow \tau_{\le m+1}( A \langle T_1, \ldots, T_n \rangle ) \]
	is strong.
	Therefore, the proof is complete.
\end{proof}

\begin{cor} \label{cor:pullback_finite_disk_projection}
	Let $f \colon X \to Y$ be a finite map of derived $k$-affinoid spaces.
	Set $A \coloneqq \Gamma(\cO\alg_Y)$ and $B \coloneqq \Gamma(\cO\alg_X)$.
	Then the square
	\[ \begin{tikzcd}
		A \arrow{r} \arrow{d} & A \langle T_1, \ldots, T_n \rangle \arrow{d} \\
		B \arrow{r} & B \langle T_1, \ldots, T_n \rangle
	\end{tikzcd} \]
	is a pushout square in $\CRing_k$.
\end{cor}

\begin{proof}
	Let
	\[ R \coloneqq A \langle T_1, \ldots, T_n \rangle \otimes_A B . \]
	\Cref{prop:relative_disk_flat_II} implies that both maps $B \to R$ and $B \to B \langle T_1, \ldots, T_n \rangle$ are flat.
	It follows from \cite[Lemma 5.45]{Porta_Yu_Representability} that the canonical map
	\[ R \longrightarrow B \langle T_1, \ldots, T_n \rangle \]
	is strong.
	It is therefore enough to prove that it induces an isomorphism on $\pi_0$.
	
	Since $\pi_0(B)$ is finite as $\pi_0(A)$-module, we have a canonical equivalence
	\[ \pi_0(A) \langle T_1, \ldots, T_n \rangle {\cotimes}_{\pi_0(A)} \pi_0(B) \simeq \pi_0(A) \langle T_1, \ldots, T_n \rangle \otimes_{\pi_0(A)} \pi_0(B) , \]
	which implies the statement.
\end{proof}

\begin{cor} \label{cor:pullback_closed_immersion_global_section}
	In the context of \cref{prop:pullback_finite_map}, assume furthermore that $f$ is a closed immersion.
	Then the canonical map
	\[ \eta \colon A' \otimes_A B \longrightarrow B \]
	is an equivalence.
\end{cor}

\begin{proof}
	Using \cite[Lemma 5.46]{Porta_Yu_Representability}, we can factor $v$ as
	\[ \begin{tikzcd}
		Y' \arrow[hook]{r}{v'} & Y \times \bD^n_k \arrow{r}{p} & Y ,
	\end{tikzcd} \]
	where $v'$ is a closed immersion and $p$ is the canonical projection.
	We can therefore decompose the given pullback square into
	\[ \begin{tikzcd}
		X' \arrow{r} \arrow{d} & X \times \bD^n_k \arrow{r} \arrow{d} & X \arrow{d} \\
		Y' \arrow{r} & Y \times \bD^n_k \arrow{r} & Y .
	\end{tikzcd} \]
	It is enough to prove that the result holds for the two squares separately.
	The proof for left square has already been dealt with in \cref{lem:global_section_pullback_double_closed_immersions}, while the proof for the right square follows from \cref{cor:pullback_finite_disk_projection}, because closed morphisms are in particular finite ones.
\end{proof}

We are now ready for the proof of \cref{prop:pullback_finite_map}.

\begin{proof}[Proof of \cref{prop:pullback_finite_map}]
	Using \cite[Lemma 5.46]{Porta_Yu_Representability} we can factor $f$ as
	\[ \begin{tikzcd}
		X \arrow[hook]{r}{f'} & Y \times \bD^n_k \arrow{r}{p} & Y ,
	\end{tikzcd} \]
	where $f'$ is a closed immersion and $p$ is the canonical projection.
	We can therefore decompose the given pullback diagram as follows:
	\[ \begin{tikzcd}
			X' \arrow{r}{u} \arrow{d} & X \arrow[hook]{d}{f'} \\
			Y' \times \bD^n_k \arrow{r} \arrow{d} & Y \times \bD^n_k \arrow{d}{p} \\
			Y' \arrow{r}{v} & Y .
	\end{tikzcd} \]
	It is enough to prove that the result holds separately for the two squares.
	The lower square has been dealt with in \cref{cor:pullback_finite_disk_projection}, while the upper square has been dealt with in \cref{cor:pullback_closed_immersion_global_section}.
	Thus, the proof is complete.
\end{proof}

\section{Projection formula}

The goal of this section is to prove a projection formula for coherent sheaves on derived \kanal spaces (cf.\ \cref{thm:projection_formula}).

\begin{defin} \label{def:coherent_cohomological_dimension}
	Let $p \colon X \to Y$ be a morphism of derived \kanal stacks.
	\begin{enumerate}
		\item We say that $p$ is \emph{proper} if its truncation $\trunc(p)$ is proper in the sense of \cite[Definition 4.8]{Porta_Yu_Higher_analytic_stacks_2014}.
		\item We say that $p$ has \emph{coherent cohomological dimension less than $n$} for an integer $n\ge 0$ if for every $\cF \in \Coh^{\ge 0}(X)$, every $i<-n$, we have $\pi_i(p_* \cF) = 0$.
		We say that $p$ has finite coherent cohomological dimension if it has coherent cohomological dimension less than an integer $n\ge 0$.
	\end{enumerate}
\end{defin}

\begin{lem}
	Let $p \colon X \to Y$ be a morphism of derived \kanal stacks.
	Assume that $p$ is proper and has coherent cohomological dimension less than $n$.
	Then the functor of stable $\infty$-categories $p_* \colon \cO_X \Mod \to \cO_Y \Mod$ restricts to a functor
	\[ p_* \colon \Coh^+(X) \to \Coh^+(Y). \]
\end{lem}

\begin{proof}
	Let $\cF \in \Coh^+(X)$.
	The hypothesis guarantees the convergence of the spectral sequence
	\[ E^{ij}_2 = \pi_i\left( p_*( \pi_j(\cF) ) \right) \Longrightarrow \pi_{i+j} \left( p_* \cF \right) . \]
	This implies that if $\cF \in \Coh^{\ge m}(X)$, then $\pi_i(p_*(\cF)) = 0$ for $i < m - n$.
	Furthermore, since $p_*(\pi_j(\cF)) \in \Cohb(Y)$  by \cite[Theorem 5.20]{Porta_Yu_Higher_analytic_stacks_2014}, we see that each $\pi_i( p_*( \pi_j(\cF) ) )$ is coherent.
	It follows that $p_*(\cF) \in \Coh^+(X)$.
\end{proof}

\begin{lem}
	Let $p \colon X \to Y$ be a morphism of derived \kanal stacks.
	\begin{enumerate}
		\item The morphism $p$ has finite coherent cohomological dimension less than $n$ if and only if for every $\cF \in \Cohh(X)$, every $i < -n$, we have $\pi_i(\rR p_* \cF) = 0$.
		\item If $p$ is representable by derived \kanal spaces whose truncations are \kanal spaces, then $p$ has finite coherent cohomological dimension.
	\end{enumerate}
\end{lem}

\begin{defin}
	An $\infty$-site $(\cC, \tau)$ is said to be \emph{quasi-compact} if for every $U \in \cC$ and every $\tau$-covering $\{U_i \to U\}_{i \in I}$ there exists a finite subset $I' \subset I$ such that $\{U_i \to U\}_{i \in I'}$ is again a $\tau$-covering.
\end{defin}

\begin{lem} \label{lem:qcompact_site}
	Let $(\cC, \tau)$ be a quasi-compact $\infty$-site.
	Then the inclusion
	\[ i_\cC \colon \Sh_{\DAb^{\le 0}}(\cC, \tau) \hookrightarrow \PSh_{\DAb^{\le 0}}(\cC) \]
	commutes with filtered colimits.
\end{lem}

\begin{proof}
	Let $F \colon J \to \PSh_{\cT}(\cC)$ be a filtered diagram and assume that for each $\alpha \in J$ the presheaf $\cF_\alpha \coloneqq F(\alpha)$ is a sheaf.
	Let $\cF$ denote the colimit of the diagram $F$, computed in $\PSh_{\DAb^{\le 0}}(\cC)$.
	To prove the lemma, it is enough to check that $\cF$ is a sheaf.
	
	Let $U \in \cC$ be a fixed object and let $\{U_i \to U\}_{i \in I}$ be a $\tau$-covering.
	Since $(\cC, \tau)$ is quasi-compact, we can assume without loss of generality that $I$ is finite.
	Let $U^\bullet$ denote the \v{C}ech nerve of this covering.
	We have to prove that the canonical map
	\[ \cF(U) \longrightarrow \lim_{n \in \mathbf \Delta} \cF(U^n) \]
	is an equivalence.
	
	To see this, it is sufficient to verify that for every $m \le 0$, the induced map
	\[ \tau_{\ge m} \cF(U) \longrightarrow \tau_{\ge m} \left( \lim_{n \in \mathbf \Delta} \cF(U^n) \right) \simeq \lim_{n \in \mathbf \Delta} \tau_{\ge m} \cF(U^n) \]
	is an equivalence.
	The spectral sequence of \cite[1.2.2.14]{Lurie_Higher_algebra} implies that the canonical map
	\[ \lim_{n \in \mathbf \Delta} \tau_{\ge m} \cF(U^n) \longrightarrow \lim_{n \in \mathbf \Delta^{\le n+2}} \tau_{\ge m} \cF(U^n) \]
	is an equivalence.
	The statement now follows because filtered colimits are $t$-exact and commute with finite limits.
\end{proof}

\begin{lem} \label{lem:proper_commutes_with_filtered}
	Let $p \colon X \to Y$ be a proper morphism of derived \kanal stacks.
	Assume that $p$ has finite coherent cohomological dimension.
	Let $F \colon J \to \Coh^{\ge 0}(X)$ be a filtered diagram.
	Assume that the colimit of $F$ exists in $\Coh^{\ge 0}(X)$ and that it is preserved by the inclusion $\Coh^{\ge 0}(X) \hookrightarrow \cO_X \Mod$.
	Then the canonical morphism
	\[ \colim_J p_* \circ F \longrightarrow p_*( \colim_J F ) \]
	is an equivalence.
\end{lem}

\begin{proof}
	The question is local on $Y$, so we can assume that $Y$ is derived affinoid.
	Since filtered colimits are $t$-exact, they commute with Postnikov towers (see the proof of \cite[Lemma 7.2]{Porta_Yu_Higher_analytic_stacks_2014}).
	In particular, we have
	\[ p_*\big(\colim_{\alpha \in J} \cF_\alpha\big) \simeq p_*\big(\colim_{\alpha \in J}\big( \lim_{n \ge 0} \tau_{\le n} \cF_\alpha\big)\big) \simeq \lim_{n \ge 0} p_*\big(\colim_{\alpha \in J} \tau_{\le n} \cF_\alpha\big) . \]
	On the other hand, consider the map
	\[ \colim_{\alpha \in J} p_*(\cF_\alpha) \simeq \colim_{\alpha \in J} \big( \lim_{n \ge 0} p_*( \tau_{\le n} \cF_\alpha ) \big) \xrightarrow{\phi} \lim_{n \ge 0} \colim_{\alpha \in J} p_*(\tau_{\le n} \cF_\alpha) . \]
	We claim that $\phi$ is an equivalence.
	
	Choose an integer $m \ge 0$ such that for every $\cG \in \Coh^{\ge 0}(X)$,
	\[ p_*(\cG) \in \Coh^{\ge m}(Y) . \]
	Then, the descent spectral sequence of \cite[Theorem 8.8]{Porta_Yu_Higher_analytic_stacks_2014} guarantees that for every $\alpha \in J$, the canonical map
	\[ \pi_i\left( p_*( \tau_{\le n + 1} \cF_\alpha ) \right) \to \pi_i \left( p_*( \tau_{\le n} \cF_\alpha) \right) \]
	is an isomorphism for every $i < n-m$.
	Moreover, since filtered colimits are $t$-exact, they commute with homotopy groups.
	It follows that the map
	\[ \pi_i\left( \colim_{\alpha \in J} p_*( \tau_{\le n + 1} \cF_\alpha ) \right) \to \pi_i \left( \colim_{\alpha \in J} p_*( \tau_{\le n} \cF_\alpha) \right) \]
	is an equivalence as well.
	This implies that $\phi$ induces an equivalence on each $\pi_i$ and hence it is an equivalence.
	
	We are therefore reduced to prove that for every $n \ge 0$, the map
	\[ \colim_{\alpha \in J} p_*( \tau_{\le n} \cF_\alpha) \longrightarrow p_*\big( \colim_{\alpha \in J} \tau_{\le n} \cF_\alpha\big) \]
	is an equivalence.
	In other words, we can assume from the very beginning that $F$ takes values in $\Coh^{\ge 0}(X) \cap \Coh^{\le n}(Y)$ for some $n \ge 0$.
	In this case, the hypothesis guarantees that $p_* \circ F$ takes values in $\Coh^{\le n}(Y) \cap \Coh^{\ge -m}(Y)$.
	
	Let us denote by $(\Geom_{/X}^{\mathrm{qc}})_{\bP_\mathrm{sm}}$ the full subcategory of the overcategory $\St(\dAfd_k,\allowbreak \tauet)_{/X}$ spanned by smooth morphisms from quasi-compact geometric stacks.
	The topology $\tauet$ restricts to a quasi-compact Grothendieck topology on $(\Geom_{/X}^{\mathrm{qc}})_{\bP_\mathrm{sm}}$.
	Notice that the inclusion $((\dAfd_{/X})_{\bP_{\mathrm{sm}}}, \tauet) \hookrightarrow ((\Geom_{/X}^{\mathrm{qc}})_{\bP_{\mathrm{sm}}}, \tauet)$ induces an equivalence of $\infty$-categories of sheaves by \cite[Proposition 2.22]{Porta_Yu_Higher_analytic_stacks_2014}.
	
	Since the colimit of $F$ can be computed in $\cO_X \Mod$, we can simply consider the sheaves $\cF_\alpha$ as sheaves on $(\Geom_{/X}^{\mathrm{qc}})_{\bP_\mathrm{sm}}$ with values in $\DAb^{\le n}$.
	Similarly, we can consider the sheaves $p_*(\cF_\alpha)$ as sheaves on $(\Geom_{/Y}^{\mathrm{qc}})_{\bP_\mathrm{sm}}$ with values in $\DAb^{\le n}$.
	Since $p$ is proper, pulling back along $p$ produces a morphism of sites
	\[\big((\Geom_{/Y}^{\mathrm{qc}})_{\bP_\mathrm{sm}}, \tauet\big) \to \big((\Geom_{/X}^{\mathrm{qc}})_{\bP_\mathrm{sm}}, \tauet\big) . \]
	In turn, this induces the following commutative diagram:
	\[ \begin{tikzcd}
		\Sh_{\DAb^{\le n}}\big((\Geom_{/X}^{\mathrm{qc}})_{\bP_\mathrm{sm}}, \tauet\big) \arrow{r} \arrow{d} & \PSh_{\DAb^{\le n}}\big((\Geom_{/X}^{\mathrm{qc}})_{\bP_\mathrm{sm}}\big) \arrow{d} \\
		\Sh_{\DAb^{\le n}}\big((\Geom_{/Y}^{\mathrm{qc}})_{\bP_\mathrm{sm}}, \tauet\big) \arrow{r} & \PSh_{\DAb^{\le n}}\big((\Geom_{/Y}^{\mathrm{qc}})_{\bP_\mathrm{sm}}\big) .
	\end{tikzcd} \]
	Since the two sites are quasi-compact, \cref{lem:qcompact_site} implies that the horizontal morphisms commute with filtered colimits.
	Since the right vertical morphism commutes with arbitrary colimits, we conclude that the left vertical morphism commutes with filtered colimits as well.
\end{proof}

\begin{thm} \label{thm:projection_formula}
	Let $p \colon X \to Y$ be a morphism of derived \kanal stacks.
	Assume that $p$ is proper and has finite coherent cohomological dimension.
	Then for any $\cF \in \Coh^+(X)$ and $\cG \in \Coh^+(Y)$, the canonical map
	\[ \eta_{\cF, \cG} \colon p_*( \cF ) \otimes_{\cO_Y} \cG \longrightarrow p_*( \cF \otimes_{\cO_X} p^*(\cG) ) \]
	is an equivalence.
\end{thm}

\begin{proof}
	The statement is local on $Y$, so we can assume that $Y$ is derived affinoid.
	By \cref{cor:geometric_realization_almost_perfect}, the inclusion $\Coh^{\ge 0}(X) \hookrightarrow \cO_X \Mod$ commutes with geometric realizations.
	
	It follows from \cref{lem:proper_commutes_with_filtered} that $p_* \colon \Coh^+(X) \to \Coh^+(Y)$ commutes with all the filtered colimits that exist in $\Coh^{\ge 0}(X)$ and that are preserved by the inclusion $\Coh^{\ge 0}(X) \hookrightarrow \cO_X \Mod$.
	Moreover, since $p_*$ is an exact functor between stable $\infty$-categories, we conclude that it commutes with all the colimits that exist in $\Coh^{\ge 0}(X)$ and that are preserved by $\Coh^{\ge 0}(X) \hookrightarrow \cO_X \Mod$.
	The same goes for $p^* \colon \Coh^-(Y) \to \Coh^-(X)$ and for the functors $\cF \otimes_{\cO_X} -$ and $p_*(\cF) \otimes_{\cO_Y} -$.
	
	Therefore, by \cref{cor:derived_affinoid_presentation_of_coherent}, we are reduced to check that $\eta_{\cF, \cG}$ is an equivalence in the special case where $\cG$ is free of finite rank over $Y$.
	In this case, the statement is tautological.
\end{proof}

\section{Proper base change}

The goal of this section is to prove the proper base change theorem for coherent sheaves on derived \kanal stacks (cf.\ \cref{thm:proper_base_change}).

We start with the following special case, and later deduce more general cases.

\begin{lem} \label{lem:proper_base_change_affinoid}
	Let
	\[ \begin{tikzcd}
		X' \arrow{r}{u} \arrow{d}{g} & X \arrow{d}{f} \\
		Y' \arrow{r}{v} & Y
	\end{tikzcd} \]
	be a pullback square of derived $k$-affinoid spaces.
	Set $A \coloneqq \Gamma(\cO\alg_Y)$, $A' \coloneqq \Gamma(\cO\alg_{Y'})$, $B \coloneqq \Gamma(\cO\alg_X)$, and $B' \coloneqq \Gamma(\cO\alg_{X'})$.
	If $v$ is finite, then for any $\cF \in \Coh^+(X)$, the canonical map
	\[ \Gamma(\cF) \otimes_A A' \to \Gamma(u^*(\cF)) \]
	is an equivalence in $A' \Mod$.
\end{lem}

\begin{proof}
	\Cref{prop:pullback_finite_map} implies that there is a canonical equivalence $A' \otimes_A B \simeq B'$ in the $\infty$-category $\CRing_k$.
	On the other hand, \cref{thm:derived_Kiehl} provides an equivalence
	\[ \Coh^+(X) \simeq \Coh^+(B) , \]
	induced by the global section functor.
	Set $M \coloneqq \Gamma(\cF)$.
	Then the chain of natural equivalences
	\[ M \otimes_A A' \simeq M \otimes_B B \otimes_A A' \simeq M \otimes_B B' \]
	implies the statement.
\end{proof}

\begin{lem} \label{lem:finite_map_pushforward}
	Let $f \colon X \to Y$ be a finite morphism of derived \kanal stacks.
	Then $f$ induces a conservative and $t$-exact functor
	\[ f_* \colon \Coh^+(X) \to \Coh^+(Y) . \]
\end{lem}

\begin{proof}
	Both statements can be checked locally on $Y$.
	So we can assume that $Y$ is derived affinoid.
	Then $X$ is also derived affinoid.
	In this case, the statement follows directly from \cref{thm:derived_Kiehl}.
\end{proof}

\begin{lem} \label{lem:finite_coherent_cohomological_dimension_pullback}
	Let
	\[ \begin{tikzcd}
	X' \arrow{r}{u} \arrow{d}{g} & X \arrow{d}{f} \\
	Y' \arrow{r}{v} & Y
	\end{tikzcd} \]
	be a pullback square of derived \kanal stacks.
	Assume that $v$ is a finite morphism and $f$ has coherent cohomological dimension less than $d$.
	Then $g$ also has coherent cohomological dimension less than $d$.
\end{lem}
\begin{proof}
	Assume that $f$ has coherent cohomological dimension less than $d$.
	Let $\cF \in \Coh^{\ge 0}(X')$.
	\Cref{lem:finite_map_pushforward} implies that $u_*(\cF) \in \Coh^{\ge 0}(X)$.
	In particular, for $i < -d$, we have
	\[ \pi_i( v_* g_*(\cF) ) \simeq \pi_i( f_* u_*(\cF) ) \simeq 0 . \]
	The statement now follows from the fact that $v_*$ is conservative and $t$-exact, as guaranteed by \cref{lem:finite_map_pushforward}.
\end{proof}

\begin{lem} \label{lem:proper_base_change_along_finite_maps}
	In the context of \cref{lem:finite_coherent_cohomological_dimension_pullback}, assume furthermore that $f$ is proper.
	Then for any $\cF \in \Coh^+(X)$ the natural map
	\[ v^* f_*(\cF) \to g_* u^*(\cF) \]
	is an equivalence.
\end{lem}
\begin{proof}
	The statement is local on $Y$.
	So we can assume that $Y$ is derived affinoid.
	Then $Y'$ is also derived affinoid.
	It is enough to check that the map $v^* f_*(\cF) \to g_* u^*(\cF)$ induces isomorphisms
	\[ \pi_i( v^*f_*(\cF)  ) \longrightarrow \pi_i( g_* u^*(\cF) ) \]
	for every $i \in \mathbb Z$.
	Fix $i \in \mathbb Z$.
	We claim that there are canonical isomorphisms
	\begin{align}
		\pi_i( v^* f_*(\cF) ) \simeq \pi_i( v^*f_*( \tau_{\le i-d} \cF ) ), \label{eq:proper_base_change_reduction_1}\\
		\pi_i( g_* u^*(\cF) ) \simeq \pi_i( g_* u^*( \tau_{\le i -d} \cF ) ). \label{eq:proper_base_change_reduction_2}
	\end{align}
	We start by showing the isomorphism \eqref{eq:proper_base_change_reduction_1}.
	Consider the fiber sequence
	\begin{equation} \label{eq:truncation_fiber_sequence_proper_base_change}
		\tau_{\ge i-d+1} \cF \to \cF \to \tau_{\le i -d} \cF
	\end{equation}
	in $\Coh^+(X)$.
	Since $f$ has coherent cohomological dimension less than $d$, we have $f_*( \tau_{\ge i-d+1} \cF ) \in \Coh^{\ge i+1}(Y)$.
	By the left $t$-exactness of $v^*$, we get
	\[ v^* f_*( \tau_{\ge i-d+1} \cF ) \in \Coh^{\ge i+1}(Y) . \]
	Finally, since both $f_*$ and $v^*$ are exact functors of stable $\infty$-categories, we have a fiber sequence
	\[ v^* f_*( \tau_{\ge i-d+1} \cF ) \to v^* f_* \cF \to v^* f_*( \tau_{\le i -d} \cF ) \]
	in $\Coh^+(Y')$, which implies the desired result.
	
	Let us now prove the the isomorphism \eqref{eq:proper_base_change_reduction_2}.
	Using once again the fiber sequence \eqref{eq:truncation_fiber_sequence_proper_base_change} and the exactness of $g_*$ and $u^*$, we obtain the fiber sequence
	\[ g_* u^*( \tau_{\ge i-d+1} \cF ) \to g_* u^*(\cF) \to g_* u^*( \tau_{\le i-d} \cF ) \]
	in $\Coh^+(Y')$.
	Since $u^*$ is left $t$-exact, we have $u^*( \tau_{\ge i-d+1} \cF ) \in \Coh^{\ge i -d + 1}(X)$.
	Moreover, since the coherent cohomological dimension of $g_*$ is less than $d$ by \cref{lem:finite_coherent_cohomological_dimension_pullback}, we deduce that
	\[ g_* u^*( \tau_{\ge i-d+1} \cF) \in \Coh^{\ge i+1}(X) . \]
	Therefore, the fiber sequence above implies the isomorphism \eqref{eq:proper_base_change_reduction_2}.
	
	It follows that we can replace $\cF$ by $\tau_{\le m} \cF$ for $m \in \mathbb Z$.
	In other words, we can assume from the beginning that $\cF$ belongs to $\Cohb(X)$.
	Moreover, since $g$ has coherent cohomological dimension less than $d$ by \cref{lem:finite_coherent_cohomological_dimension_pullback}, we also have an isomorphism
	\[ \pi_i( g_* u^*(\cF) ) \simeq \pi_i( g_*(\tau_{\le i-d}u^*(\cF)) ) . \]
	Therefore, it is now sufficient to check that the composition
	\[ v^* f_*( \cF ) \to g_* u^*( \cF ) \to g_* (\tau_{\le i -d} u^*(\cF) ) \]
	induces an isomorphism on $\pi_i$ for every $i \in \mathbb Z$ and every $\cF \in \Cohb(X)$.
	
	At this point, by the properness of $f$, we can find an étale hypercovering $U^\bullet$ of $X$ such that each $U^m$ is a disjoint union of finitely many derived $k$-affinoid spaces.
	Let $V^\bullet \coloneqq X' \times_X U^\bullet$.
	Let $u^n \colon V^n \to U^n$, $f^n \colon U^n \to Y$ and $g^n \colon V^n \to Y'$ be the canonical maps.
	Using \cite[Corollary 8.6]{Porta_Yu_Higher_analytic_stacks_2014} and the fact that each $g^m_*$ is $t$-exact, we obtain equivalences
	\[ f_*( \cF ) \simeq \lim_{m \in \mathbf \Delta} f^m_*( \cF |_{U^m} ), \qquad \pi_i( g_* u^*(\cF) ) \simeq \pi_i \left( \lim_{m \in \mathbf \Delta} g^m_*\big( \tau_{\le i -d} u^*(\cF) |_{V^m} \big)  \right) . \]
	Since $f$ and $g$ have finite coherent cohomological dimension and $\cF \in \Cohb(X)$, the above limits are limits of diagrams taking values in $\Coh^{[a,b]}(Y')$, for suitable $a, b \in \mathbb Z$.
	As this is a $(b-a)$-category, combining \cref{prop:n_cofinal_colimit} and \cref{cor:Delta_cofinal} we can replace $\mathbf \Delta$ by $\mathbf \Delta_{\le l}$ for some $l \gg 0$ in the above cosimplicial diagrams.
	Therefore, by the exactness of $v^*$, we have
	\[ v^* f_*(\cF) \simeq \lim_{m \in \mathbf \Delta_{\le l}} v^* f^m_*(\cF |_{U^m}) . \]
	
	Since each $U^m$ and each $V^m$ is derived affinoid, by \cref{lem:proper_base_change_affinoid}, we have the following canonical equivalence:
	\[ v^* f^m_*(\cF |_{U^m}) \simeq g^m_* u^{m*}( \cF |_{U^m} ) \simeq g^m_* (u^* \cF) |_{V^m} . \]
	Therefore, we obtain
	\begin{align*}
		\pi_i(v^* f_*(\cF)) & \simeq \pi_i \left( \lim_{m \in \mathbf \Delta_{\le l}} g^m_* ( u^* \cF )|_{V^m} \right) \\
		& \simeq \pi_i \left( \lim_{m \in \mathbf \Delta_{\le l}} g_*^m \big( \tau_{\le i -d} u^* (\cF) \big) \right) \\
		& \simeq \pi_i \left( g_* u^*(\cF) \right)
	\end{align*}
	So the proof is complete.
\end{proof}

We are now ready to state and prove the main result of this section.

\begin{thm} \label{thm:proper_base_change}
	Let
	\[ \begin{tikzcd}
	X' \arrow{r}{u} \arrow{d}{g} & X \arrow{d}{f} \\
	Y' \arrow{r}{v} & Y
	\end{tikzcd} \]
	be a pullback square of derived \kanal stacks.
	Assume that $f$ is proper and has finite coherent cohomological dimension.
	Then for any $\cF \in \Coh^+(X)$, the natural morphism
	\[ v^* f_*(\cF) \to g_* u^*(\cF) \]
	is an equivalence.
\end{thm}

\begin{proof}
	We prove the theorem by reduction to the case treated in \cref{lem:proper_base_change_along_finite_maps}.
	The statement being local on $Y$ and $Y'$, we can assume that both are derived affinoid.
	Now we proceed in the following steps:
	
	\medskip
	\paragraph{Step 1. Reduction to the case where $Y'$ is discrete.}
	Let $\cG \coloneqq \fib( v^* f_*(\cF) \to g_* u^*(\cF))$ and let $j \colon \trunc(Y') \to Y$ be the canonical closed immersion.
	We claim that if $j^* \cG = 0$ then $\cG = 0$.
	Suppose to the contrary that $\cG \ne 0$.
	Since $\cG$ is coherent, we can choose a maximal $i$ such that $\pi_i(\cG) \ne 0$.
	Then
	\[ \pi_i(j^*(\cG)) = \pi_i(\cG \otimes_{\cO\alg_V} \pi_0(\cO\alg_V)) = \pi_i(\cG) \otimes_{\pi_0(\cO\alg_V)} \pi_0(\cO\alg_V) = \pi_i(\cG) \ne 0, \]
	which is a contradiction.
	So we have shown that $j^* \cG = 0$ implies $\cG = 0$.
	Therefore we can replace $Y'$ by $\trunc(Y')$.
	In other words, we are reduced to the case where $Y'$ is discrete.
	
	\medskip
	\paragraph{Step 2. Reduction to the case where $Y$ is discrete.}
	Since $Y'$ is discrete, the map $Y' \to Y$ factors as $Y' \to \trunc(Y) \to Y$.
	Since the map $\trunc(Y) \to Y$ is a closed immersion, by \cref{lem:proper_base_change_along_finite_maps}, the statement of the theorem holds for the pullback square
	\[ \begin{tikzcd}
	X_0 \arrow{r}{i} \arrow{d}{f_0} & X \arrow{d}{f} \\
	\trunc(Y) \arrow{r}{j} & Y .
	\end{tikzcd} \]
	By \cref{lem:finite_coherent_cohomological_dimension_pullback}, we see that $f_0$ has finite coherent cohomological dimension.
	Therefore, we can replace $Y$ by $\trunc(Y)$ and $f$ by $f_0$.
	In other words, we are reduced to the case where $Y$ is discrete.
		
	\medskip
	\paragraph{Step 3. Reduction to the case $Y' = \Sp K$, where $K$ is a finite field extension of $k$.}
	Assume that the theorem holds in the case $Y' = \Sp K$, where $K$ is a finite field extension of $k$.
	Suppose to the contrary that there is an affinoid $Y' = \Sp A$ and a coherent sheaf $\cF \in \Coh^+(X)$  for which the morphism
	\[ v^* f_*(\cF) \to g_* u^*(\cF) \]
	is not an equivalence.
	Let $\cG \coloneqq \fib( v^* f_*(\cF) \to g_* u^*(\cF) )$.
	Since $\cG$ is coherent, we can select the minimal integer $i$ such that $\pi_i(\cG) \ne 0$.
	Since $\pi_i(\cG)$ is coherent, there exists a closed point $p \colon \Sp K \to Y'$ such that $p^*(\pi_i(\cG)) \ne 0$.
	It follows from \cite[7.2.1.23]{Lurie_Higher_algebra} that $\pi_i( p^* \cG) = \pi_0(p^*( \pi_i(\cG) )) \ne 0$.
	Note that
	\[ p^* \cG \simeq \fib( p^* v^* f_*(\cF) \to p^* g_* u^*(\cF) ) .  \]
	Consider the diagram of pullback squares
	\[ \begin{tikzcd}
		X_p \arrow{r}{q} \arrow{d}{g'} & X' \arrow{r}{u} \arrow{d}{g} & X \arrow{d}{f} \\
		\Sp K \arrow{r}{p} & Y' \arrow{r}{v} & Y .
	\end{tikzcd} \]
	Since $p$ is a closed immersion, \cref{lem:proper_base_change_along_finite_maps} implies that the canonical map
	\[ p^* g_* u^*(\cF) \to g'_* q^* u^*(\cF) \]
	is an equivalence.
	It follows that we can identify $p^* \cG$ with the fiber of
	\[ (v \circ p)^* f_*(\cF) \to g'_* (u \circ q)^*(\cF) . \]
	Since $p^* \cG$ is non-zero, this contradicts our assumption that the theorem holds for $Y' = \Sp K$.
	We conclude that it is now enough to prove the theorem in the special case $Y'=\Sp K$.
	
	\medskip
	\paragraph{Step 4.}
	Let $y$ the image of $v \colon \Sp K \to Y$.
	Let $\kappa(y)$ be the residue field of $Y$ at $y$.
	We factor $v$ as
	\[ \Sp K \to \Sp(\kappa(y)) \xrightarrow{j} Y , \]
	where $j$ is a closed immersion.
	Consider the diagram of pullback squares
	\[ \begin{tikzcd}
	X' \arrow{r} \arrow{d} & X_y \arrow{r} \arrow{d} & X \arrow{d} \\
	\Sp K \arrow{r} & \Sp(\kappa(y)) \arrow{r}{j} & Y .
	\end{tikzcd} \]
	The statement of the theorem holds for the square on the right because $j$ is a closed immersion.
	On the other hand, note that $K$ is a finite field extension of $\kappa(y)$.
	In particular, $\Sp K \to \Sp(\kappa(y))$ is a finite map.
	Therefore, \cref{lem:proper_base_change_along_finite_maps} implies that the statement of the theorem holds also for the square on the left.
	We deduce that the statement holds for the outer square too, completing the proof of the theorem.
\end{proof}

\section{The plus pushforward} \label{sec:plus_pushforward}

The goal of this section is to introduce the plus pushforward functor on perfect complexes.

\begin{defin}
	Let $X$ be a derived \kanal stack and let $\cF \in \Coh^+(X)$.
	We say that $\cF$ has tor-amplitude less than $d$ if there exists a derived $k$-affinoid atlas $\{U_i\}$ of $X$ such that $\Gamma(\cF|_{U_i})$ has tor-amplitude less than $d$ as a $\Gamma(\cO\alg_{U_i})$-module.
\end{defin}

\begin{defin}
	We say that a morphism $f \colon X \to Y$ of derived \kanal spaces has \emph{tor-amplitude less than $d$} for an integer $d\ge 0$ if there exists a derived $k$-affinoid étale covering $\{V_i\}$ of $Y$ and derived $k$-affinoid étale coverings $\{U_{ij}\}$ of $V_i \times_Y X$ such that every morphism $\Gamma(\cO_{V_i}\alg) \to \Gamma(\cO_{U_{ij}}\alg)$ of simplicial commutative $k$-algebras has tor-amplitude less than $d$.
	We say that $f \colon X \to Y$ has \emph{finite tor-amplitude} if it has tor-amplitude less than $d$ for some integer $d \ge 0$.
\end{defin}

\begin{lem} \label{lem:tor_amplitude_truncations}
	Let $X$ be a derived \kanal stack and let $j \colon \trunc(X) \hookrightarrow X$ be the embedding from its truncation.
	Then any $\cF \in \Coh^+(X)$ has finite tor-amplitude if and only if $j^*(\cF)$ has finite tor-amplitude.
\end{lem}

\begin{proof}
	The statement being local on $X$, we can assume $X$ to be derived affinoid.
	Then we conclude from \cref{thm:derived_Kiehl} and \cite[Proposition 2.22(3)]{Toen_Moduli}.
\end{proof}

\begin{lem}
	Let $f \colon X \to Y$ be a morphism of derived \kanal stacks.
	Let $j \colon \trunc(Y) \hookrightarrow Y$ denote the canonical embedding.
	Then $f$ has finite tor-amplitude if and only if the base change $X \times_Y \trunc(Y) \to \trunc(Y)$ has finite tor-amplitude.
\end{lem}

\begin{proof}
	The question is local on both $X$ and $Y$, so we can assume that both are derived $k$-affinoid.
	Since $j \colon \trunc(Y) \to Y$ is a closed immersion, the hypotheses of \cref{prop:pullback_finite_map} are satisfied.
	The lemma is then a direct consequence of \cref{lem:tor_amplitude_truncations}.
\end{proof}

\begin{defin}
	Let $X$ be a derived \kanal stack.
	We say that $\cF \in \Coh^+(X)$ is a \emph{perfect complex} if there exists a derived $k$-affinoid atlas $\{U_i\}$ of $X$ such that $\Gamma(U_i; \cF|_{U_i})$ is perfect as $\Gamma(\cO\alg_{U_i})$-module in the sense of \cite[7.2.4.1]{Lurie_Higher_algebra}.
\end{defin}

\begin{lem} \label{lem:perfect_vs_finite_tor_amplitude}
	Let $X$ be a derived \kanal stack and let $\cF \in \Coh^+(X)$.
	Then $\cF$ is perfect if and only if it has finite tor-amplitude.
\end{lem}

\begin{proof}
	This is a direct consequence of \cite[7.2.4.23(4)]{Lurie_Higher_algebra}.
\end{proof}

\begin{lem} \label{lem:proper_finite_cohomological_dimension_preserves_almost_perfect}
	Let $f \colon X \to Y$ be a proper morphism of derived \kanal stacks.
	Assume that $f$ has finite coherent cohomological dimension.
	Then
	\[ f_* \colon \cO_X \Mod \to \cO_Y \Mod \]
	restricts to a functor
	\[ f_* \colon \Coh^+(X) \to \Coh^+(Y) . \]
\end{lem}

\begin{proof}
	First of all, we remark that the cohomological descent spectral sequence provided by \cite[Theorem 8.8]{Porta_Yu_Higher_analytic_stacks_2014} and the fact that $f$ has finite coherent cohomological dimension imply that it is enough to prove that $f_* \colon \cO_X \Mod \to \cO_Y \Mod$ takes $\Cohh(X)$ to $\Coh^+(X)$.
	
	Next, we observe that \cite[Remark 2.1.5]{DAG-VIII} implies that pushforward along the canonical closed immersion $j \colon \trunc(X) \to X$ induces an equivalence
	\[ \Cohh(\trunc(X)) \simeq \Cohh(X) . \]
	Finally, consider the commutative diagram
	\[ \begin{tikzcd}
		\trunc(X) \arrow{r}{j} \arrow{d}{\trunc(f)} & X \arrow{d}{f} \\
		\trunc(Y) \arrow{r}{i} & Y .
	\end{tikzcd} \]
	Let $\cF \in \Cohh(X)$ and choose $\cG \in \Cohh(\trunc(X))$ such that $j_*(\cG) \simeq \cF$.
	Then
	\[ f_*(\cF) \simeq f_*(j_*(\cG)) \simeq i_*(\trunc(f)_*(\cG)) . \]
	Since $i_*$ takes $\Coh^+(\trunc(Y))$ to $\Coh^+(Y)$, it is enough to prove that $\trunc(f)_*(\cG) \in \Coh^+(X)$.
	Since $\trunc(f)$ is proper, this follows from \cite[Theorem 5.20]{Porta_Yu_Higher_analytic_stacks_2014}.
\end{proof}

\begin{prop} \label{prop:perfect_pushforward}
	Let $f \colon X \to Y$ be a proper morphism of derived \kanal stacks.
	Assume that $f$ has finite coherent cohomological dimension and finite tor-amplitude.
	Then $f_* \colon \Coh^+(X) \to \Coh^+(Y)$ preserves perfect complexes.
\end{prop}

\begin{proof}
	The statement being local on $Y$, we can assume $Y$ to be derived affinoid.
	
	Let $\cF \in \Coh^+(X)$ be a perfect complex.
	Since $f$ is proper and has finite coherent cohomological dimension, \cref{lem:proper_finite_cohomological_dimension_preserves_almost_perfect} implies that $f_*(\cF) \in \Coh^+(X)$.
	Combining \cref{thm:derived_Kiehl} and \cref{lem:perfect_vs_finite_tor_amplitude}, we are reduced to check that $f_*(\cF)$ has finite tor-amplitude.
	Let $j \colon \trunc(Y) \hookrightarrow Y$ denote the canonical embedding.
	Then in virtue of \cref{lem:tor_amplitude_truncations}, it is enough to prove that $j^*(f_*(\cF))$ has finite tor-amplitude.
	Consider the pullback square
	\[ \begin{tikzcd}
		X_0 \arrow{r}{i} \arrow{d}{f_0} & X \arrow{d}{f} \\
		\trunc(Y) \arrow{r} & Y .
	\end{tikzcd} \]
	Using the proper base change (cf.\ \cref{thm:proper_base_change}), we obtain a canonical equivalence
	\[ j^*(f_*(\cF)) \simeq f_{0*}(i^*(\cF)) . \]
	As $i^*(\cF)$ is again perfect, we can replace $Y$ by $\trunc(Y_0)$.
	In other words, we can assume $Y$ to be truncated from the very beginning.
	
	In this case, since $f$ has finite tor-amplitude and is proper, the structure sheaf of $X$ is bounded above.
	In other words, there exists an integer $n \ge 0$ such that $\cO\alg_X \in \Coh^{\le n}(X)$.
	Since $f$ is proper and $Y$ is quasi-compact, $X$ is quasi-compact as well.
	Therefore, we obtain that $\cF \in \Cohb(X)$.
	Using again the properness of $f$, we find an étale hypercovering $U^\bullet$ of $X$ such that each $U^n$ is a disjoint union of finitely many derived $k$-affinoid spaces.
	Then \cite[Corollary 8.6]{Porta_Yu_Higher_analytic_stacks_2014} implies that
	\[ f_*(\cF) \simeq \lim_{m \in \mathbf \Delta} f^m_*( \cF |_{U^m} ) . \]
	Since $f$ has finite coherent cohomological dimension and each map $f^m_* \colon \Coh^+(U^m) \to \cO_Y \Mod$ is $t$-exact, we see that the above limit takes values in the $n$-category $\cO_Y \Mod^{\le n} \cap \cO_Y \Mod^{\ge d}$ for some $d \ll 0$.
	Set $l \coloneqq n - d +2$.
	The $(n-d)$-cofinality of $\mathbf \Delta_{\le n-d+2} \hookrightarrow \mathbf \Delta$ implies that we can rewrite the above limit as
	\[ f_*(\cF) \simeq \lim_{m \in \mathbf \Delta_{\le l}} f^m_*(\cF  |_{U^m}) . \]
	Since $f$ has finite tor-amplitude, each $U^m$ has finite tor-amplitude over $Y$.
	Therefore, each $f^m_*(\cF|_{U^m})$ has finite tor-amplitude.
	The proposition now follows from the fact that the full subcategory of $\cO_Y \Mod$ spanned by modules of finite tor-amplitude is closed under finite limits.
\end{proof}

For any derived \kanal stacks, any $\cF \in \Perf(X)$, we define
\[ \cF^\vee \coloneqq \cHom_{\cO\alg_X}(\cF, \cO\alg_X) . \]

\begin{defin} \label{def:plus_pushforward}
	Let $f \colon X \to Y$ be a proper morphism of derived \kanal stacks.
	Assume that $f$ has finite coherent cohomological dimension and finite tor-amplitude.
	For any $\cF \in \Perf(X)$, we define
	\[ f_+(\cF) \coloneqq (f_*(\cF^\vee))^\vee. \]
\end{defin}

Let $f \colon X \to Y$ be as in \cref{def:plus_pushforward}.
Then for any $\cF \in \Perf(X)$, we have a canonical map
\begin{equation} \label{eq:eta_F}
	\eta_\cF \colon \cF \longrightarrow f^*(f_+(\cF)).
\end{equation}
Indeed, the counit of the adjunction $f^* \dashv f_*$ gives a map
\[ f^* f_*(\cF^\vee) \to \cF^\vee.\]
Since $\cF$ and $f^* f_*(\cF^\vee)$ are perfect, the map above induces a map $\cF \to ( f^* f_*(\cF^\vee) )^\vee$.
Then by the equivalence $( f^* f_*(\cF^\vee) )^\vee \simeq f^*((f_*(\cF^\vee))^\vee)$, we obtain the canonical map \eqref{eq:eta_F}.

\begin{prop}
	Let $f \colon X \to Y$ be a proper morphism of derived \kanal stacks.
	Assume that $f$ has finite coherent cohomological dimension and finite tor-amplitude.
	Then for any $\cF \in \Perf(X)$ and $\cG \in \Coh^+(Y)$, the map $\eta_\cF \colon \cF \to f^*(f_*(\cF))$ induces an equivalence
	\[ \Map_{\Coh^+(Y)}(f_+(\cF), \cG) \simeq \Map_{\Coh^+(X)}(\cF, f^*(\cG)) . \]
\end{prop}

\begin{proof}
	It is enough to prove that $\eta_\cF$ induces an equivalence
	\[ \cHom_{\cO\alg_Y}(f_+(\cF), \cG) \simeq f_* \cHom_{\cO\alg_X}(\cF, f^*(\cG)) . \]
	Since $\cF$ is perfect, we have
	\[ \cHom_{\cO\alg_X}(\cF, f^*(\cG)) \simeq \cF^\vee \otimes_{\cO\alg_X} f^*(\cG) . \]
	Under this identification, the statement follows from the following chain of equivalences
	\[ f_*(\cF^\vee \otimes_{\cO\alg_X} f^*(\cG)) \simeq f_*(\cF^\vee) \otimes \cG \simeq \cHom_{\cO\alg_Y}(f_*(\cF^\vee)^\vee, \cG), \]
	where the first equivalence is the projection formula from \cref{thm:projection_formula}.
\end{proof}

\begin{prop} \label{prop:plus_pushforward_base_change}
	Let
	\[ \begin{tikzcd}
		X' \arrow{r}{u} \arrow{d}{g} & X \arrow{d}{f} \\
		Y' \arrow{r}{v} & Y
	\end{tikzcd} \]
	be a pullback square of derived \kanal stacks.
	Assume that $f$ is proper, has finite coherent cohomological dimension and finite tor-amplitude.
	Then for any $\cF \in \Perf(X)$, the canonical map
	\[ g_+(u^*(\cF)) \longrightarrow v^*(f_+(\cF)) \]
	is an equivalence.
\end{prop}

\begin{proof}
	It follows from the proper base change of \cref{thm:proper_base_change} together with the explicit formula for the plus pushforward.
\end{proof}

\section{Derived mapping stack} \label{sec:mapping_stack}

In this section, we prove the representability of the $\infty$-functor $\bfMap_S(X,Y)$ in \cref{thm:Map_intro}, the main theorem of this paper.
Our approach is based on the representability theorem in derived analytic geometry, which has been established in our previous work \cite{Porta_Yu_Representability}.
We refer to Section 7 in loc.\ cit.\ for the criteria used in the representability theorem, namely, the notions of infinitesimally cartesian, convergent, and having a global analytic cotangent complex.

First let us give a functorial description of the $\infty$-functor $\bfMap_S(X,Y)$ in full generality.

Let $F \to S$ be a morphism of stacks over the site $(\dAfd_k,\tauet)$.
Since the overcategory $\St(\dAfdk, \tauet)_{/S}$ is an $\infty$-topos, the functor
\[ F \times_S - \colon \St(\dAfdk, \tauet)_{/S} \longrightarrow \St(\dAfdk, \tauet)_{/S} \]
given by product with $F$ commutes with colimits.
So it admits a right adjoint, which we denote by
\[ \bfMap_S(F,-) \colon \St(\dAfdk, \tauet)_{/S} \longrightarrow \St(\dAfdk, \tauet)_{/S} . \]
Notice that for every $G \in \St(\dAfdk, \tauet)_{/S}$ and every derived $k$-affinoid $T$ equipped with a map $T \to S$, the Yoneda lemma implies that
\begin{align*}
	\bfMap_S(F,G)(T) & \simeq \Map_{\St(\dAfdk, \tauet)_{/S}}(T, \bfMap_S(F,G)) \\
	& \simeq \Map_{\St(\dAfdk, \tauet)_{/S}}(T \times_S F, G) \\
	& \simeq \Map_{\St(\dAfdk, \tauet)_{/T}}(T \times_S F, T \times_S G) .
\end{align*}
So it coincides with the description in the introduction.

\begin{lem} \label{lem:map_preserves_inf_cartesian}
	Let $G \to S$ be an infinitesimally cartesian morphism of stacks over the site $(\dAfdk,\tauet)$.
	Then for any $F \in \St(\dAfdk, \tauet)_{/S}$, the morphism $\bfMap_S(F,G) \to S$ is infinitesimally cartesian.
\end{lem}

\begin{proof}
	Thanks to \cite[6.1.3.9]{HTT}, we can assume $S$ to be a derived $k$-affinoid space.
	Let $V \in \dAfd_{k / S}$, $\cF \in \Coh^{\ge 1}(V)$ and let $d \colon \anL_V \to \cF$ be an $S$-linear analytic derivation.
	Let $V_d[\cF]$ be the square-zero extension determined by $d$.
	We have to prove that the diagram
	\[ \begin{tikzcd}
		\bfMap_S(F,G)(V_d[\cF]) \arrow{r} \arrow{d} & \bfMap_S(F,G)(V) \times_{\bfMap_S(F,G)(V[\cF])} \bfMap_S(F,G)(V) \arrow{d} \\
		S(V_d[\cF]) \arrow{r} & S(V) \times_{S(V[\cF])} S(V)
	\end{tikzcd} \]
	is a pullback square.
	In this case, the bottom arrow is an equivalence thanks to \cite[Lemma 7.7]{Porta_Yu_Representability}.
	We are therefore reduced to check that the map
	\[ \bfMap_S(F,G)(V_d[\cF]) \longrightarrow \bfMap_S(F,G)(V) \times_{\bfMap_S(F,G)(V[\cF])} \bfMap_S(F,G)(V) \]
	is an equivalence.
	
	Unraveling the definitions, we can rewrite the above map as
	\[ \Map_{/S}(F \times_S V_d[\cF], G) \longrightarrow \Map_{/S}(F \times_S V, G) \times_{\Map_{/S}(F \times_S V[\cF], G)} \Map_{/S}(F \times_S V, G) , \]
	and again as
	\begin{multline*}
	\Map_{/S}\big(F, \bfMap_S(V_d[\cF], G)\big) \to \\
	\Map_{/S}(F, \bfMap_S(V, G)) \times_{\Map_{/S}(F, \bfMap_S(V[\cF], G))} \Map_{/S}\big(F, \bfMap_S(V, G)\big) .
	\end{multline*}
	Since the functor $\Map_S(F,-)$ commutes with limits, it is enough to prove that the canonical map
	\[ \bfMap_S(V_d[\cF], G) \to \bfMap_S(V, G) \times_{\bfMap_S(V[\cF], G)} \bfMap_S(V, G) \]
	is an equivalence.
	In other words, we have to prove that for every $U \in \dAfd_{k/S}$, the map
	\[ \Map(U \times_S V_d[\cF], G) \to \Map(U \times_S V, G) \times_{\Map(U \times_S V[\cF], G)} \Map(U \times_S V, G) \]
	is an equivalence.
	Let $p \colon U \times_S V \to V$ and $q \colon U \times_S V \to U$ be the canonical projections, and let $\pi \colon U \times_S V \to S$ be the canonical morphism.
	Then we have $U \times_S V[\cF] \simeq (U \times_S V)[p^* \cF]$, and $U \times_S V_d[\cF] \simeq (U \times V)_{p^*(d)}[p^* \cF]$, where $p^*(d)$ is the analytic derivation
	\[ \anL_{U \times_S V} \simeq p^* \anL_V \oplus_{\pi^* \anL_S} q^* \anL_U \to p^* \anL_V \to p^* \cF . \]
	The lemma now follows from the fact that the morphism $G \to S$ is infinitesimally cartesian.
\end{proof}

\begin{lem} \label{lem:map_convergent}
	Let $G \to S$ be a convergent morphism in $\St(\dAfd_k,\tauet)$.
	Then for any $F \in \St(\dAfdk, \tauet)_{/S}$, the morphism $\bfMap_S(F,G) \to S$ is convergent.
\end{lem}

\begin{proof}
	Thanks to \cite[6.1.3.9]{HTT}, we can assume $S$ to be a derived $k$-affinoid space.
	Let $V \in \dAfdk$ be a derived $k$-affinoid space over $S$.
	We have to check that the canonical map
	\[ \bfMap_S(F,G)(V) \longrightarrow \lim_n \bfMap_S(F,G)(\mathrm t_{\le n} V) \]
	is an equivalence.
	Reasoning as in the proof of \cref{lem:map_preserves_inf_cartesian}, we can reduce ourselves to the case where $F$ is a derived $k$-affinoid space over $S$.
	In this case, it suffices to check that the map
	\[ \Map_{/S}(F \times_S V, G) \longrightarrow \lim_n \Map_{/S}(F \times_S \mathrm t_{\le n} V, G) \]
	is an equivalence.
	Consider the commutative diagram
	\[ \begin{tikzcd}
		\Map_{/S}(F \times_S V, G) \arrow{d}{g_1} \arrow{r}{f_1} & \lim_m \Map_{/S}(\mathrm t_{\le m}(F \times_S V), G) \arrow{d}{g_2} \\
		\lim_n \Map_{/S}(F \times_S \mathrm t_{\le n} V, G) \arrow{r}{f_2} & \lim_{m,n} \Map(\mathrm t_{\le m}( F \times_S \mathrm t_{\le n} V), G).
	\end{tikzcd} \]
	Since the morphism $G \to S$ is convergent, the morphisms $f_1$ and $f_2$ are equivalences.
	Moreover, for every fixed $n$ and every $m \ge n$, the map
	\[ \mathrm t_{\le m}(F \times_S \mathrm t_{\le n} V) \longrightarrow \mathrm t_{\le m}(F \times_S V) \]
	is an equivalence.
	It follows that $g_2$ is an equivalence.
	Therefore, $g_1$ is an equivalence as well.
\end{proof}

\begin{defin}
	A derived \kanal stack $F$ is said to be \emph{locally of finite presentation} if its analytic cotangent complex is perfect.
\end{defin}

\begin{lem} \label{lem:map_cotangent_complex}
	Let $f \colon X \to S$ be a morphism in $\St(\dAfdk, \tauet)$ representable by proper flat derived \kanal spaces.
	Let $g \colon Y \to S$ be a morphism in $\St(\dAfdk, \tauet)$ representable by derived \kanal stacks locally of finite presentation.
	Then the canonical map $\bfMap_S(X,Y) \to S$ admits a relative analytic cotangent complex which is furthermore perfect.
\end{lem}
\begin{proof}
	The statement is local on $S$, so we can assume $S$ to be derived affinoid.
	Consider the canonical maps
	\[ \begin{tikzcd}
		X \times_S \bfMap_S(X, Y) \arrow{r}{\mathrm{ev}} \arrow{d}{\pi} & Y \\
		\bfMap_S(X, Y). & 
	\end{tikzcd} \]
	Since $Y$ is locally of finite presentation, its cotangent complex $\anL_{Y/S}$ is perfect.
	Since $f$ is representable by proper flat derived \kanal spaces, so is $\pi$.
	We can therefore set
	\[ \cF \coloneqq \pi_+ ( \mathrm{ev}^*( \anL_{Y/S} ) ) .  \]
	\Cref{prop:perfect_pushforward} implies that $\cF$ is a perfect complex on $\bfMap_S(X,Y)$.
	We claim that it is the relative cotangent complex of the map $\bfMap_S(X,Y) \to S$.
	
	Fix a map $u \colon U \to \bfMap_S(X,Y)$ from a derived $k$-affinoid space.
	Denote by $f_u \colon X \times_S U \to Y$ the morphism classified by $u$.
	Notice that we have a pullback square
	\[ \begin{tikzcd}
		X \times_S U \arrow{r}{q} \arrow{d}{\pi_u} & X \times_S \bfMap_S(X,Y) \arrow{d}{\pi} \\
		U \arrow{r}{u} & \bfMap_S(X,Y),
	\end{tikzcd} \]
	and that $f_u \simeq \mathrm{ev} \circ q$.
	
	Let $\cG \in \Coh^{\ge 0}(U)$.
	Unraveling the definitions, we see that
	\begin{multline*}
	\DerAn_{\bfMap_S(X,Y) / S}(U, \cG) \simeq \Map_{\Coh^+(X \times_S U)}(f_u^* \anL_{Y/S}, \pi_u^*(\cG))\\
	\simeq \Map_{\Coh^+(U)}(\pi_{u+} f_u^* \anL_{Y/S}, \cG) .
	\end{multline*}
	The base change property of $\pi_+$ given by \cref{prop:plus_pushforward_base_change} implies that
	\[ \pi_{u+} f_u^* \anL_{Y/S} \simeq u^*( \pi_+ \mathrm{ev}^*( \anL_{Y/S} ) ) . \]
	It follows that $\pi_+ \mathrm{ev}^*( \anL_{Y/S} )$ is the analytic cotangent complex of the map
	\[\bfMap_S(X,Y) \to S.\]
\end{proof}

\begin{prop} \label{prop:Map_conditional}
	Let $S$ be a rigid \kanal space.
	Let $X$ be a proper flat rigid \kanal space over $S$.
	Let $Y$ be a derived \kanal stack locally of finite presentation over $S$.
	Assume that $\trunc(\bfMap_S(X,Y))$ is representable by a rigid \kanal stack.
	Then $\bfMap_S(X,Y)$ is representable by a derived \kanal stack locally of finite presentation over $S$.
\end{prop}

\begin{proof}
	This follows from Lemmas \ref{lem:map_preserves_inf_cartesian}, \ref{lem:map_convergent} and \ref{lem:map_cotangent_complex} and from \cite[Theorem 7.1]{Porta_Yu_Representability}.
\end{proof}

\begin{thm} \label{thm:Map}
	Let $S$ be a rigid \kanal space.
	Let $X, Y$ be rigid \kanal spaces over $S$.
	Assume that $X$ is proper and flat over $S$, and that $Y$ is separated over $S$.
	Then the $\infty$-functor $\bfMap_S(X,Y)$ is representable by a derived \kanal space separated over $S$.
\end{thm}

\begin{proof}
	For any rigid $k$-affinoid space $T$, we have
	\begin{multline*}
	\bfMap_S(X,Y)(T) \simeq \Map_{\St(\dAfdk, \tauet)_{/S}}(X \times T, Y) \simeq \Hom_{\An_{k/S}}(X \times T, Y)\\
	\simeq \bfHom_S(X,Y)(T) .
	\end{multline*}
	In other words, the truncation $\trunc \bfMap(X,Y)$ is equivalent to the ordinary Hom functor from $X$ to $Y$ over $S$.
	By \cite[Proposition 5.3.3]{Conrad_Spreading-out}, the functor $\bfHom_S(X,Y)$ is representable by a rigid \kanal space separated over $S$.
	Therefore, it follows from \cref{prop:Map_conditional} that $\bfMap(X,Y)$ is representable by a derived \kanal space separated over $S$.
\end{proof}

\section{Appendix: $n$-cofinality}

The $\infty$-categorical version of cofinality differs significantly from the classical one.
Indeed, let $f \colon \cC \to \cD$ be a functor between $\infty$-categories.
Then \cite[4.1.3.1]{HTT} implies that $f$ is cofinal if and only if $\cC \times_{\cD} \cD_{D/}$ is weakly contractible for every $D \in \cD$.
On the other hand, we know that if $\cC$ and $\cD$ are $1$-categories, then it is enough to ask $\cC \times_{\cD} \cD_{D/}$ to be non-empty and weakly connected.
For example, the inclusion $\mathbf \Delta_{\le 2}\op \hookrightarrow \mathbf \Delta\op$ is cofinal in the classical sense but not in the $\infty$-categorical one.

In between there the cases where $\cC \times_{\cD} \cD_{D/}$ is weakly $n$-connected for some integer $n \ge 2$.
There is a corresponding intermediate notion of $n$-cofinality, which is known among the experts but that has never appeared in the literature.
The goal of this appendix is to introduce this notion and prove its basic properties.

\begin{defin}
	A functor $p \colon \cC \to \cD$ of $\infty$-categories is \emph{$n$-cofinal} if for any right fibration $\cE \to \cD$ with $n$-truncated fibers the map
	\[ \Map_\cD(\cD, \cE) \to \Map_\cD(\cC, \cE) \]
	is a homotopy equivalence.
\end{defin}

\begin{prop} \label{prop:n_cofinal_colimit}
	Let $p \colon \cC \to \cD$ be an $n$-cofinal functor between $\infty$-categories.
	Then for any $(n+1)$-category $\cE$ and any diagram $f \colon \cD \to \cE$, the induced map $\cE_{f/} \to \cE_{f \circ p/}$ is an equivalence of $\infty$-categories.
\end{prop}

\begin{proof}
	We follow the proof of \cite[4.1.1.8]{HTT}.
	By \cite[4.1.1.5]{HTT}, for every $x \in \cE$, we have identifications
	\[ \cE_{f/} \times_\cE \{x\} \simeq \Map_\cE(\cD, \cE_{/x}), \qquad \cE_{f \circ p /} \times_\cE \{x\} \simeq \Map_\cE(\cC, \cE_{/x}) . \]
	Recall from \cite[Proposition 8.2]{Porta_Yu_DNAnG_I} that $\cE_{/x}$ is an $n$-category.
	In other words, the projection $\cE_{/x} \to \cE$ is a right fibration with $n$-truncated fibers.
	It follows that the natural map
	\[ \cE_{f/} \longrightarrow \cE_{p \circ f/} \]
	induces an equivalence of the fibers over every object $x \in \cE$.
	The conclusion now follows because both $\cE_{f/}$ and $\cE_{p \circ f/}$ are left fibered over $\cE$ (see \cite[2.4.7.12]{HTT}).
\end{proof}

\begin{prop} \label{prop:n_connected_Quillen_Theorem_A}
	Let $f \colon \cC \to \cD$ be a functor between $\infty$-categories.
	Suppose that for every object $D \in \cD$, the $\infty$-category $\cC \times_\cD \cD_{D/}$ is weakly $(n+1)$-connected.
	Then $f$ is $n$-cofinal.
\end{prop}

\begin{proof}
	The same proof of \cite[4.1.3.1]{HTT} shows that it is enough to prove the following statement: if $f \colon \cC \to \cD$ is a cartesian fibration of $\infty$-categories whose fibers are weakly $(n+1)$-connected, then $f$ is $n$-cofinal.
	The following is a model-independent formulation of the proof given in \cite[4.1.3.2]{HTT} (in the case $n = \infty$).
	
	Let $p \colon \cE \to \cD$ be a right fibration with $n$-truncated fibers.
	Consider the functor
	\[ F \colon (\Cat_\infty)_{/\cD} \to \Cat_\infty \]
	informally given by
	\[ F(\cT) \coloneqq \Fun_\cD(\cC \times_\cD \cT, \cE) . \]
	Since $\Cat_\infty$ is cartesian closed, the functor $F$ commutes with colimits.
	Therefore, the adjoint functor theorem implies that $F$ is representable by an object in $(\Cat_\infty)_{/\cD}$, which we denote by $\bfMap_\cD(\cC, \cE)$.
	Let
	\[ q \colon \bfMap_\cD(\cC, \cE) \longrightarrow \cD \]
	be the structural morphism.
	We claim that $q$ is a cartesian fibration.
	Let $a \colon x \to y$ be an arrow in $\cD$.
	We can represent the fiber of $q$ at $y$ as the $\infty$-category of functors from $\cC_y$ to $\cE_y$.
	Notice that $\cE_y$ is an $\infty$-groupoid.
	This implies that
	\[ \Fun(\cC_y, \cE_y) \simeq \Fun(\mathrm{EnvGpd}(\cC_y), \cE_y) , \]
	where $\mathrm{EnvGpd}(\cC_y)$ denotes the enveloping $\infty$-groupoid of $\cC_y$, in other words, $\mathrm{EnvGpd}(\cC_y)$ is the $\infty$-groupoid obtained by $\cC_y$ by inverting \emph{all} arrows of $\cC_y$.
	Now recall that $\cE_y$ is $n$-truncated and that $\mathrm{EnvGpd}(\cC_y)$ is $(n+1)$-connected.
	It follows that there is an equivalence
	\[ \Fun(\mathrm{EnvGpd}(\cC_y), \cE_y) \simeq \cE_y , \]
	induced by sending $s \colon \cC_y \to \cE_y$ to $s(z)$ for a fixed object $z \in \cC_y$.
	Let now
	\[ a^*_\cC \colon \cC_y \to \cC_x , \qquad a_\cE^* \colon \cE_y \to \cE_x  \]
	be the induced functors.
	Note that $\Fun(\cC_y, \cE_x)$ is categorically equivalent to a point.
	So the evaluation at $a^*_\cC(z)$ induces a diagram
	\[ \begin{tikzcd}
		\cC_y \arrow{rr}{s(z)} \arrow{d}{a^*_\cC} & &  \cE_y \arrow{d}{a^*_\cE} \\
		\cC_x \arrow{rr}{a^*_\cE(s(z))} & & \cE_x ,
	\end{tikzcd} \]
	which is commutative up to homotopy;
	moreover, the commutative diagram above exhibits the constant functor associated to $a^*_\cE(s(z))$ as a cartesian lift of $a \colon x \to y$.
	This proves that $q$ is a cartesian fibration.
	
	The description of the $p$-cartesian edges shows that the associated functor is equivalent to the functor classified by $\cE$.
	Therefore, $q$ is equivalent, as cartesian fibration, to $p \colon \cE \to \cD$.
	In particular, taking $\cD$-sections, we obtain that the canonical morphism
	\[ \Map_\cD(\cC, \cE) \simeq \Map_\cD(\cD, \cE) \]
	is an equivalence.
	In other words, $f \colon \cC \to \cD$ is $n$-cofinal.
\end{proof}

\begin{cor} \label{cor:Delta_cofinal}
	The inclusion $\mathbf \Delta_{\le n+2}\op \subseteq \mathbf \Delta\op$ is $n$-cofinal.
\end{cor}

\bibliographystyle{plain}
\bibliography{dahema}

\end{document}